\documentclass[10 pt,reqno]{amsart}
\usepackage{latexsym, amsmath, amssymb, longtable, booktabs,amscd,microtype,booktabs,cases}
\usepackage[square]{natbib}
\usepackage{graphicx}
\usepackage[dvipsnames]{xcolor}
\usepackage{caption}
\usepackage{tcolorbox}
\usepackage{dsfont}
\usepackage[all]{xy}
\usepackage{enumerate}
\usepackage{multirow}
\usepackage{tikz}
\usetikzlibrary{matrix,arrows,decorations.pathmorphing}
\usepackage{collectbox}
\usepackage[OT2,T1]{fontenc}
\DeclareSymbolFont{cyrletters}{OT2}{wncyr}{m}{n}
\DeclareMathSymbol{\Sha}{\mathalpha}{cyrletters}{"58}

\usepackage{amsxtra}

\usepackage[colorinlistoftodos,prependcaption, textsize=tiny]{todonotes}
\reversemarginpar

\usepackage{enumitem}

\usepackage{hyperref}
\hypersetup{
	colorlinks,
	citecolor=blue,
	filecolor=black,
	linkcolor=red,
	urlcolor=black
}
\bibpunct{[}{]}{,}{n}{}{;} 

\usepackage{mathtools}

\DeclarePairedDelimiterX{\infdivx}[2]{}{}{%
	#1\;\delimsize\|\;#2%
}

\newcommand\Frob{{\mathrm{Frob}}}

\newcommand{\Z}{\mathbb{Z}}
\newcommand{\F}{\mathbb{F}}
\newcommand\SL{{\mathrm {SL}}}
\newcommand{\Ga}{\Gamma}
\newcommand{\ga}{\gamma}
\newcommand{\de}{\delta}
\newcommand{\De}{\Delta}

\newcommand{\La}{\Lambda}

\newcommand{\lra}{\longrightarrow}

\newcommand{\C}{\mathbb{C}}

\newcommand\GL{{\mathrm{GL}}}
\newcommand\Hom{\mathrm{Hom}}
\newcommand\HH{{\mathrm H}}

\newcommand{\Q}{\mathbb{Q}}

\newcommand\Gal{{\mathrm{Gal}}}

\newcommand\Ind{{\mathrm{Ind}}}

\newcommand{\cyc}{\mathrm{cyc}}

\makeatother

\newtheorem{thm}{Theorem}[section]
\newtheorem{theorem}[thm]{Theorem}
\newtheorem{cor}[thm]{Corollary}

\newtheorem{prop}[thm]{Proposition}
\newtheorem{lemma}[thm]{Lemma}
\theoremstyle{definition}
\newtheorem{definition}[thm]{Definition}
\newtheorem{remark}[thm]{Remark}
\newtheorem{example}[thm]{Example}

\theoremstyle{definition}

\theoremstyle{remark}

\theoremstyle{remark}

\makeatletter
\def\imod#1{\allowbreak\mkern10mu({\operator@font mod}\,\,#1)}
\makeatother

\title{Multiplicities in Selmer groups 
and root numbers for Artin twists}
\author{Somnath Jha, Tathagata Mandal and Sudhanshu Shekhar}
\address{Somnath Jha, jhasom@iitk.ac.in, 
	IIT Kanpur \& 
	IIT Goa, India}
\address{Tathagata Mandal, math.tathagata@gmail.com, 
	IIT Kanpur, India}
\address{Sudhanshu Shekhar, sudhansh@iitk.ac.in, 
	IIT Kanpur, India}

\begin{document}
\begin{abstract}
Let  $K/F$ be a finite Galois extension of number fields and $\sigma$ be an absolutely irreducible, self-dual representation of $\Gal(K/F)$. Let $p$ be an odd prime and consider two elliptic curves $E_1, E_2$  with good, ordinary reduction at primes  above $p$ and equivalent mod-$p$  Galois representations. In this article, we study the variation of the parity of the multiplicities of $\sigma$ in the representation space 
associated to the $p^\infty$-Selmer group of $E_i$ over $K$. We also compare the root numbers for the twist of $E_i/F$ by $\sigma$ 
and show that the $p$-parity conjecture holds for the twist of $E_1/F$ by $\sigma$ if and only if it holds for the twist of $E_2/F$ by $\sigma$.  We also express Mazur-Rubin-Nekov\'a\v{r}'s arithmetic local constants in terms of certain local Iwasawa invariants. 
\end{abstract}
	
	\maketitle

	\section*{Introduction}
	In this article,
	we study the variation of the $p$-parity conjecture  for Artin twists of two elliptic curves over certain number field, that have equivalent mod-$p$ Galois representations.
	\vskip 1mm
	Let  $E$ be an elliptic curve over a number field $K$. Let 
	$W(E/K) \in \{ \pm 1\}$  be the (global) root number (\S \ref{s7}) given by the sign of the expected functional equation of the complex $L$-function  $L(E/K,s)$.  
	Given a rational prime $p$ and a number field $K$, recall the definition of $p^\infty$-Selmer group of $E$ over $K$:
	\begin{small}{\begin{equation} \label{sel0}
			0 \longrightarrow S_p(E/K) \longrightarrow \HH^1(K, E[p^\infty]) 
			{\longrightarrow} \prod_{\text{all prime }u \text{ in } K}^{} \HH^1(K_u,E),
			\end{equation}}\end{small}
	\noindent with $E[p^\infty]:= \underset{{{}n \geq 1}}{\cup}E[p^n](\overline{{{}K}})$.   The Pontryagin dual $S_p(E/K)^\vee:=\mathrm{Hom}_\mathrm{cont}(S_p(E/K), \Q_p/{\Z_p}) $ is a finitely generated $\Z_p$ module and we define $s_p:= \mathrm{rank}_{\Z_p} \ S_p(E/K)^\vee= \mathrm{dim}_{\Q_p} \ S_p(E/K)^\vee \otimes_{\Z_p} \Q_p$. Then the $p$-parity conjecture states that $W(E/K)=(-1)^{s_p}.$

	\vskip 1mm
	
	\par The $p$-parity conjecture is known  for elliptic curves defined over many number fields, due to works of many authors including Dokchitser-Dokchitser, Monsky,  Mazur-Rubin, Nekov\'a\v{r} etc. (cf. \cite{MR2534092}). However for an arbitrary number field $F$, { the} $p$-parity conjecture is not established yet. 
	Now, let $K/F$ be a finite Galois extension of number fields. Then  the complex $L$-function $L(E/K, s)$ can be expressed as a product of  $L(E/F, \sigma, s)$, where $\sigma$ varies over the irreducible Artin representations of Gal$(K/F)$, each with multiplicity deg$(\sigma)$. A variant of the {{}Birch and Swinnerton-Dyer} (BSD) conjecture,  relates for every such $\sigma$, the multiplicity  of $\sigma$ in $E(K)\otimes_\Z \C$ with $\text{ord}_{s=1} ~ L(E/F,\sigma, s)$ (see \cite[\S 4]{ICM}). Thus one naturally considers the $p$-parity conjecture for the twists of $E/F$ by $\sigma$;  relating the multiplicity of the Artin representations in the twisted $p^\infty$-Selmer groups to the root number of the twisted $L$-functions. We need a couple of definitions to state it precisely. 
	\vskip 1mm
	\par  Fixing an embedding of $\overline{\Q}$ into $\overline{\Q}_p$,  an Artin representation $\sigma$ of $\De:= \mathrm{Gal}(K/F)$ can be viewed as  $\sigma: \De \rightarrow \mathrm{GL}_n(\overline{\Q}) \hookrightarrow \mathrm{GL}_n(\overline{\Q}_p)$. For an irreducible representation $\sigma: \De \rightarrow \mathrm{GL}_n(\overline{\Q}_p)$, define $s_p(E,\sigma)$ to be the multiplicity of $\sigma$ in the $\overline{\Q}_p$ representation space   $\mathrm{Hom}_\mathrm{cont}(S_p(E/K), \Q_p/{\Z_p}) \otimes_{\Z_p} \overline{\Q}_p$. On the other hand, for any self-dual irreducible (Artin) representation $\sigma$ of $\De=\mathrm{Gal}(K/F)$, the (global) root number  $W(E/F, \sigma) \in \{ \pm 1\}$ is the sign of the expected functional equation of $L(E/F,\sigma,s)$ i.e. $W(E/F, \sigma) =(-1)^{\mathrm{ord}_{s=1} L(E/F,\sigma,s)}$ (possibly conjectural). However, in the above setting, there is a precise definition of $W(E/F, \sigma)$ due to Deligne \cite{MR0349635}, which does not depend on any conjecture. 
	\vskip 1mm
	
	The twisted $p$-parity conjecture states that for any elliptic curve $E$ defined over $F$ and any self-dual, irreducible representation $\sigma$ of $\De$,
	\begin{small}{\begin{equation}\label{sigma-parity-state} 
			W(E/F, \sigma)=(-1)^{s_p(E/F,\sigma)}.
			\end{equation}}\end{small}
	We refer to \cite[Theorem 1.3]{MR2534092}, where this twisted $p$-parity conjecture is proved for certain number fields $K$, under some hypotheses on $E$.
\par {{} Let $\mu_p$ denote the group  of $p$-th roots of unity. For any number field or an $\ell$-adic field $L$, we  will denote by $L_\cyc$ the cyclotomic $\Z_p$-extension of $L$}. The following is the main theorem (Theorem~\ref{thm4}) of this paper; it studies the variation of $p$-parity conjecture of an Artin twist of two congruent elliptic curves by comparing the ratio of the root numbers of the twists  with the ratio of the multiplicities of the twists in the Selmer groups. 
\begin{theorem}[Theorem \ref{thm4}]\label{mainthm01}
		Let $p$ be an odd prime and  $k$ be a number field such that all primes $v \mid p$ of $k$  are ramified in $k(\mu_p)/k$. Let $E_1$, $E_2$ be two elliptic curves over $k$ with good, ordinary reduction at all the primes $v \mid p$. We assume  $E_1[p] \cong E_2[p]$ as $G_k$-modules. Consider number fields $k\subset F \subset K$   with $K \cap F_\cyc=F$. Let $\sigma: \De=\mathrm{Gal}(K/F) \rightarrow \mathrm{GL}_n(\overline{\Q}_p)$ be an irreducible, self-dual, orthogonal representation.  Further assume that
		\begin{enumerate}
			\item [(i)]
			 $E_1[p](K)=0$.
			\item [(ii)]
			 $S_p(E_1/K_\cyc)[p]$ is finite. 
		\end{enumerate}
Then {\small \begin{equation} \label{p-parity}
	\frac{W(E_1/F,\sigma)}{W(E_2/F,\sigma)}=(-1)^{s_p(E_1/F,\sigma)-s_p(E_2/F,\sigma)}.
			\end{equation}}
		In particular, the $p$-parity conjecture stated in \eqref{sigma-parity-state} holds for the twist of $E_1/F$ by $\sigma$ if and only if it holds for the twist of $E_2/F$ by $\sigma$. \qed
	\end{theorem}
	Using Theorem \ref{thm4}, we give  new examples of $E, F, K$ and $\sigma$ (see Example \ref{example:1} \&  Example \ref{example2}) such that the $p$-parity conjecture in \eqref{sigma-parity-state} holds  for $(E, F, \sigma)$, under some mild condition related to certain $\mu$-invariants.  
	
	{{}In Theorem \ref{thm4}, it is assumed that $\sigma$ is orthogonal. However, as we point out in Remark \ref{dcjcjkecje},  when $\sigma$ is  symplectic, the results stated in Theorem \ref{thm4} are known to hold trivially, assuming the non-degeneracy of the $p$-adic height pairing.} 
	
	In the reverse situation, where $E$ is fixed and $\rho_1, \rho_2$ are two self-dual Artin representation of $\De$ such that $\tilde{\rho}_1 \cong \tilde{\rho}_2$ (see \S \ref{sel1}), the variation of the $p$-parity conjecture has already been discussed by Greenberg (see Remark \ref{revsit}).
	\vskip 1mm
	
	To prove Theorem \ref{thm4}, we study  the parity of ${s_p(E_1/F,\sigma)-s_p(E_2/F,\sigma)}$ and $\frac{W(E_1/F,\sigma)}{W(E_2/F,\sigma)}$ separately. It is well known (Prop. \ref{prop1}) that $s_p(E_i/F,\sigma) $ and the multiplicity of $\sigma$, $\lambda_{E_i}(\sigma)$ in $S_p(E_i[p^\infty]\otimes\sigma/F_\cyc)^\vee$ have the same parity. Keeping this in mind, we apply Iwasawa theoretic methods to compare the parity of $\lambda_{E_1}(\sigma)$ with $\lambda_{E_2}(\sigma)$ for the congruent  elliptic curves.  From this, it reduces to the comparison of the parity of local Iwasawa invariants $\delta_{E_i,v}(\sigma)$, for $v \in \Sigma_0$; a finite set  of primes of $F$, away from $p$. But  this Iwasawa theoretic set up works, only when {{} we} assume some conditions, which are stated in Hypothesis (H) in \S \ref{sel1}. On the other hand, the global root number is given by product of all the local root numbers. Using this, we compare the local root numbers $W(E_i/F_v,\sigma_v)$ at each prime $v$ and then go on to show that the ratio of the global root numbers of the Artin twists are also given by the parity of precisely the  same local Iwasawa invariants, away from $p$; thereby proving  Theorem \ref{thm4}.
	\vskip 1mm
	In fact, the comparison of the parity $s_p(E_1/F,\sigma)$ with $s_p(E_2/F,\sigma)$ is established in Theorem \ref{thm1} and the root numbers of the Artin twists were compared in Theorem \ref{propo4}. We now state  Theorem \ref{thm1}, which we think is also independently interesting. Let $\Sigma$ be a finite set of primes of $F$ containing the primes above $p$,  the infinite primes, the primes of bad reduction of $E_1$ and $E_2$ and primes which ramify in $K/F$ and $\Sigma_0$ is certain subset of $\Sigma$ away from $p$, defined in \eqref{Sigma_0}. We denote by $S_i, N_i,  W, X, Y_3, Z_3$  certain subsets of $\Sigma_0$, defined in \ref{exp-prime}.

	\begin{theorem}[Theorem \ref{thm1}] \label{thm1intro}
		Let $p$ be an odd prime and $k$ be a number field such that all primes of $k$ lying above $p$ are ramified in $k(\mu_p)/k$. Let $E_1, E_2$ be two elliptic curves over $k$ with $E_1[p] \cong E_2[p]$ as $G_k$-modules. Let $\sigma$ be an absolutely irreducible self-dual orthogonal representation of $\De$. Assume the hypothesis {\bf (H)}. Then \begin{small}{\begin{eqnarray*}
					&&s_p(E_1, \sigma) + \sum_{v \in S_1} \langle \sigma_v, 1_v \rangle 
					+
					\sum_{v \in N_1}^{} \langle \sigma_v, \varkappa_v \rangle  + \sum_{v\in W} \big( \langle \sigma_v, \vartheta_v \rangle +  \langle \sigma_v,  \omega_v\vartheta_v \rangle \big)  + \sum_{v \in Z_3} \big( \langle \sigma_v, \omega_v\vartheta_{v} \rangle +  \langle \sigma_v, \vartheta_{v}\theta_{3,v} \rangle \big)\\
					& & \equiv s_p(E_2,\sigma) + \sum_{v \in S_2} \langle \sigma_v, 1_v \rangle 
					+
					\sum_{v \in N_2}^{} \langle \sigma_v, \varkappa_v \rangle +
					\sum_{v \in X} \langle \sigma_v, \vartheta_{v} \rangle
					+ \sum_{v \in Y_3}  \big(   \langle \sigma_v,   1_v \rangle  + \langle \sigma_v,   \varkappa_v \rangle +  \langle \sigma_v,   \theta_{3,v}\rangle\big) 
					\pmod{2}. 
		\end{eqnarray*}}\end{small}
		(Here 
		$\langle \beta, \alpha \rangle$ denotes the multiplicity of $\alpha$ in $\beta$, see \S \ref{sel1} for the  definitions  of the representations involved.)  
	\end{theorem}
	We would like to mention that our proof is more involved when the prime $p=3$. In addition, we also need delicate arguments to compare the local Iwasawa invariants $\delta_{E_i,v}(\sigma)$ and local root number of Artin twists at primes $v$ of additive reductions, away from $p$. Moreover, the study of ${\delta_{E_1,v}(\sigma)}-{\delta_{E_2,v}(\sigma)}$ and ${W(E_1/F_v, \sigma_v)}/{W(E_2/F_v, \sigma_v)}$ at primes $v \nmid p$ of $F$ such that  $v \mid 2$ or $v \mid 3$ needs careful consideration. We have covered all such possibilities without any exception. 
	
	\medskip
	
	Mazur-Rubin have developed arithmetic theory of local constants which has been used extensively in Nekov\'a\v{r}'s study of the $p$-parity conjecture.  Nekov\'a\v{r}'s setting applies when we make an additional hypothesis  that the Galois isomorphism between $E_1[p]$ and $E_2[p]$ is symplectic. In this symplectic isomorphism setting,  Nekov\'a\v{r} \cite{ne} expresses the ratio of the local root numbers of two congruent mod-$p$ curves, at a prime $v \nmid p$, in terms of arithmetic local constant at that prime. Moreover, using Mazur-Rubin's results, it is possible to express the {{}S}elmer corank in terms of arithmetic local constants at various primes, including primes dividing $p$ (see \eqref{apenalgbe}). This naturally suggest{{}s} that  the parity of our local Iwasawa invariants $\delta_{E_i,v}(\sigma), v \nmid p$ are related to this arithmetic local constant. We establish this relation  by directly comparing   the difference of the local Iwasawa invariants at a prime away from $p$  with the arithmetic local constant at that prime (Proposition \ref{apencomppro}). We can think  {{}of}  this phenomenon as an Iwasawa theoretic interpretation of the arithmetic local constants.

	However, as mentioned in Remark \ref{remarklast1}, the situation at $v=p$ is more complicated (see \cite{ne} and his subsequent works) and even in the symplectic isomorphism setting, we believe Theorem \ref{thm4} does not follow only from the theory of arithmetic local constants.

	\par Our results are a generalization of results obtained in \cite{MR3503694} and \cite{MR3629245}. In particular, when $\sigma$ is the trivial representation,  \cite[Theorem  5.5]{MR3629245} follows from our Theorem \ref{thm4} and  \cite[Theorem $1.1$]{MR3503694} follows from our Theorem \ref{thm1}.  We stress that the results in \cite{MR3503694} and \cite{MR3629245} only concerns the $p$-parity conjecture for the trivial twist and  this generalization to a non-trivial twist in this article is not straightforward. The $p$-parity conjecture, for a non-trivial twist is much less understood, compared to the trivial twist case (cf. \cite[Theorem 1.3]{MR2534092}) and we undertake this study in this article. Also, we give a  direct comparison and  description of Mazur-Rubin-Nekov\'a\v{r}'s  arithmetic local constants  in terms of (a priori different) explicit  local Iwasawa invariants  (Proposition \ref{apencomppro}); this was not considered earlier. 	Further, in this article, we give   new examples of $E, F, K, \sigma$ such that the $p$-parity conjecture for the twist of $E/F$ by $\sigma$ holds, under some mild assumption related to certain $\mu$-invariants (Examples \ref{example:1} \&  \ref{example2}).	

\medskip

	The summary and structure of this article is as follows: The basic algebraic set up is discussed in section \ref{sel1} and we collect all the relevant known results related to $s_p(E/F,\sigma)$ in this section, many of which are due to Greenberg. In subsection \ref{s2.1},  we consider  imprimitive Selmer groups and  prove that the multiplicity of  $\sigma$ in the imprimitive $\Sigma_0$ Selmer group, with our suitably chosen $\Sigma_0$, coincides for the congruent mod-$p$ elliptic curves. Note that, here we make an improvement on methods of \cite{gv}, by choosing a  $\Sigma_0$  which is smaller than the corresponding set of \cite{gv}.   This reduces the study of ${s_p(E_1/F,\sigma)-s_p(E_2/F,\sigma)} \pmod 2$ to the study of  local invariants at ${\delta_{E_1,v}(\sigma)}-{\delta_{E_2,v}(\sigma)}$ for each $v \in \Sigma_0$, which we do in subsection \ref{sub1}. Combining these results, we prove Theorem \ref{thm1} in subsection \ref{mainspei}, which completes the comparison of $s_p(E_i/F,\sigma)$.   Subsection \ref{s5} contains basic known results on root numbers, many of them are due to Rohrlich. We compare the root numbers $\frac{W(E_1/F,\sigma)}{W(E_2/F,\sigma)}$ in subsection \ref{rootnumbereisigma} and express it in terms of local Iwasawa invariants ${\delta_{E_i,v}(\sigma)},$ where $ v \in \Sigma_0$. Using this, we prove our main result, Theorem \ref{thm4} in subsection \ref{finalrusyl}.  Then we  discuss numerical examples of $p$-parity conjecture for an Artin twist of an elliptic curve,  as applications of our result, in subsection \ref{examplessubsec}. Finally, in section \ref{compmrnjms}, we directly  compare local Iwasawa invariant{{}s} at a prime away from $p$ with Mazur-Rubin and Nekov\'a\v{r}'s arithmetic local constants and show our methods are consistent with the theory of arithmetic local constant{{}s}.

	\medskip
	
	\noindent {\bf Notation:} 
	{{} For a number field or an $\ell$-adic field $L$, $L_\cyc$ will be the cyclotomic $\Z_p$-extension of $L$ with $\Gamma:= \Gal(L_\cyc/L) \cong \Z_p$. Recall, the Iwasawa algebra of $\Gamma$,  denoted by $\Lambda$, is defined as $\La=\Z_p[[\Ga]]:=\underset{n}{\varprojlim} \ \Z_p[\frac{\Ga}{\Ga^{p^n}}]$. For a discrete $\La$ module $M$, we denote the Pontryagin dual $\mathrm{Hom}_{\mathrm{cont}} (M,\Q_p/{\Z_p})$ by $M^\vee$. For any separable field $L$, $G_L$ will denote the Galois group $\Gal(\overline{L}/L)$ and $I_L$ denotes the  inertia subgroup of $G_L$. For an algebraic extension $L$ of $\Q$, the completion of $L$ at a prime $v$ of $L$ is denoted by  $L_v$.}

\section{The algebraic set up}\label{sel1}
	Let $p$ be an odd prime and $\mu_p$ denotes the  group of $p$-th roots of unity. Let $k$ be a number field such that all primes of $k$ lying above $p$ are ramified in $k(\mu_p)/k$. Now $F$ be any number field containing $k$. Consider a  finite Galois extension $K/F$ such that $K \cap F_\cyc=F$, where $F_\cyc$ is the cyclotomic $\Z_p$-extension of $F$. We  set $\De:=\Gal(K/F)$.  
	
	\medskip
	
	\par Let $E$ be an elliptic curve over $k$ with good, ordinary reduction at all primes of $k$ lying above $p$. Let $N^F_E$ and $\overline{N}^F_E$ respectively denote the   conductor of $E$ and the prime-to-$p$ part of the  conductor of the $G_F$ module $E[p]$. Here {{}the} conductor of $E[p]$ is defined following Serre and can be found in \cite[\S1, Page 135]{li}.  {{} Let $T_pE$ denote} the Tate module of $E$ and $\rho_{E}: G_k \to \mathrm{Aut}(T_pE) \cong \GL_2(\Z_p)$ be the Galois representation associated to $E$.  Let  $\Sigma
	$ be a finite set of places of $F$ containing the primes of bad reduction of $E$, the primes above $p$, the infinite primes and primes that  ramify in $K/F$.  We write $F_\Sigma$ for the maximal algebraic extension of $F$ unramified outside $\Sigma$.

	\medskip

	Given an irreducible Artin representation $\sigma: \De \lra \mathrm{GL}_n(\overline{\Q}_p)$, 
	we can view $\sigma$ as an absolutely irreducible representation 
	$\sigma: \De \lra \mathrm{GL}_n(\mathcal F)$, where $\mathcal F$ is some finite extension of $\Q_p$
	with ring of integer $O$, residue field $\mathfrak{f}$ and uniformizer $\pi$. 
	Let $L_\sigma$ be a $\De$-invariant $O$-lattice in the underlying $\mathcal F$ representation space 
	$W_\sigma$ 
	for $\sigma$ and set $\text{dim}(\sigma) = \text{dim}_{\mathcal F} \ W_\sigma$. We will denote the $G_F$ module $E[p^\infty] \otimes_{\Z_p} L_{\sigma}$ by $E[p^\infty] \otimes \sigma$. 
	Reducing modulo $\pi$, we obtain a representation $\widetilde{\sigma}$ over $\mathfrak{f}$ and let $U_{\widetilde{\sigma}}$ denote the $\mathfrak{f}$-representation space 
	for $\widetilde{\sigma}$. The $G_F$ module $E[p] \otimes_{\mathbb{F}_p} U_{\widetilde{\sigma}}$ is denoted by $E[p] \otimes \widetilde{\sigma}$.  
	We  have $(E[p^\infty] \otimes \sigma)[\pi] \cong E[p] \otimes \widetilde{\sigma}$.

	\medskip
	
	\par Let  $v,u,w,\eta$ respectively denote primes in $F,K,F_\cyc, K_\cyc$ such that 
	{{} $\eta\mid u \mid v$ and $\eta \mid w \mid v$}. Let $\widetilde{E}_v$ denote the reduction of $E$ at a prime $v \mid p$ of $F$. As $E$ is  good, ordinary  at all primes $v$ above $p$,  $\widetilde{E}_v[p^\infty]:= \bigcup_n \widetilde{E}_v[p^n]$ is an unramified $G_{F_v}$-module for every $v \mid p$. Set
	\begin{small}{\begin{equation}\label{189}
			\mathcal{H}_v(F_\cyc,E[p^\infty]\otimes \sigma): =
			\begin{cases} 
			\prod_{w \mid v}^{} 
			\HH^1(F_{\cyc,w}, E[p^\infty] \otimes \sigma) & \text{if } v \nmid p \\
			\prod_{w \mid v}^{} 
			\HH^1(F_{\cyc,w}, \widetilde{E}_v[p^\infty] \otimes \sigma)
			& \text{if } v \mid  p.
			\end{cases}
			\end{equation}
			\begin{equation}\label{190}
			\mathcal{H}_v(F_\cyc,E[p]\otimes \widetilde{\sigma}): =
			\begin{cases} 
			\prod_{w \mid v}^{} 
			\HH^1(F_{\cyc,w}, E[p] \otimes \widetilde{\sigma}) & \text{if } v \nmid p \\
			\prod_{w \mid v}^{} 
			\HH^1(F_{\cyc,w}, \widetilde{E}_v[p] \otimes \widetilde{\sigma}) 
			& \text{if } v \mid  p. 
			\end{cases} 
			\end{equation}
			and 
			\begin{equation}\label{389}
			\mathcal{H}_u(K_\cyc,E): =
			\begin{cases} 
			\prod_{\eta \mid u}^{} 
			\HH^1(K_{\cyc,\eta}, E[p^\infty] ) & \text{if }  u \mid v,  u \nmid  p, v \in \Sigma.  \\
			\prod_{\eta \mid u}^{} 
			\HH^1(K_{\cyc,\eta}, \widetilde{E}_v[p^\infty] )
			& \text{if } u \mid p.
			\end{cases}
			\end{equation}
	}\end{small}
	
	\medskip
	
	For any subset $\Sigma' \subset \Sigma \setminus \{v \text{ prime in } F: v \mid p\infty\}$, we define various \emph{imprimitive} $\Sigma'$-Selmer groups as follows:
	\begin{small}{
			\begin{equation} \label{exact2}
			0 \longrightarrow S_p^{\Sigma'}(E[p^\infty] \otimes \sigma/F_\cyc) \longrightarrow
			\HH^1(F_\Sigma/F_\cyc, E[p^\infty] \otimes \sigma) 
			\overset{\phi}{\longrightarrow}
			\prod_{v \in \Sigma \setminus \Sigma'}^{} 
			\mathcal{H}_v(F_\cyc, E[p^\infty]\otimes \sigma).
			\end{equation}
			\begin{equation} \label{exact3}
			0 \longrightarrow S_p^{\Sigma'}(E[p] \otimes \widetilde{\sigma}/F_\cyc) \longrightarrow
			\HH^1(F_\Sigma/F_\cyc, E[p] \otimes \widetilde{\sigma}) 
			\longrightarrow \prod_{v \in \Sigma \setminus \Sigma'}^{} 
			\mathcal{H}_v(F_\cyc, E[p]\otimes \widetilde{\sigma}).
			\end{equation}}\end{small}
	\begin{small}{\begin{equation} \label{exact4}
			0 \longrightarrow S_p^{\Sigma'}(E/K_\cyc) \longrightarrow
			\HH^1(F_\Sigma/K_\cyc,, E[p^\infty]) 
			\longrightarrow
			\prod_{ u \mid v, v \in \Sigma \setminus  \Sigma'}^{} 
			\mathcal{H}_u(K_\cyc, E), 
			\end{equation}
	}\end{small}
We also recall that by our choice of $\Sigma$, we have 
	\begin{small}{\begin{equation} \label{exact5}
			0 \longrightarrow S_p(E/K_\cyc) \longrightarrow
			\HH^1(F_\Sigma/K_\cyc,, E[p^\infty]) 
			\longrightarrow
			\prod_{u \mid v, v \in \Sigma}^{} 
			\mathcal{H}_u(K_\cyc, E).
			\end{equation}}
	\end{small}
	
	\medskip
	
	Now assume that $E$ has good, ordinary reduction at all primes $v $ of $F$ dividing $p$ and also assume  that $S_p(E/K_\cyc)[p]$ is finite.  Then $S_p(E/K_\cyc)^\vee$ 
	is a finitely generated $\Z_p$-module.  For any irreducible representation $\sigma$ of $\De$ over $\mathcal{F}$,  we define $\lambda_E(\sigma)$  (respectively $\lambda_E^{\Sigma'}(\sigma)$) to be the multiplicity of $\sigma$ in the $\mathcal F$ representation space  $S_p(E/K_\cyc)^\vee \otimes_{\Z_p} \mathcal F$ (respectively $S^{\Sigma'}_p(E/K_\cyc)^\vee\otimes_{\Z_p} \mathcal F.$
	
	\medskip
	
\par Then under the hypotheses {{}that} $E$ is good, ordinary at primes dividing $p$ with  $S_p(E/K_\cyc)[p]$ finite,  it is shown \cite[Proposition $4.3.1$]{MR2807791} that
	\begin{small}{\begin{equation}\label{897}
			\lambda_E^{\Sigma'}(\sigma)=
			\mathrm{rank}_O \ S_p^{\Sigma'}(E[p^\infty] \otimes \sigma/F_\cyc)^\vee.
			\end{equation}}\end{small}
Moreover we have the following result of Greenberg:
\begin{lemma} \label{L2}\cite[Prop. 4.3.3]{MR2807791}
		Let $E$ be good, ordinary at primes of $F$ dividing $p$. Assume that either $E[p](K)=0$ or   $\Sigma'$ is non-empty. Also assume that $S_p(E/K_\cyc)[p]$ {{}is} finite. Then $S_p^{\Sigma'}(E[p^\infty] \otimes \sigma/F_\cyc)$ is $\pi$-divisible. In particular, {{}the} maximal pseudo-null $\Lambda$-submodule of $S_p^{\Sigma'}(E[p^\infty] \otimes \sigma/F_\cyc)^\vee$ is $0$. \qed 
	\end{lemma}
By Lemma \ref{L2},  it follows from \eqref{897} that 
	\begin{small}{\begin{equation}\label{896}
			\lambda_E^{\Sigma'}(\sigma)=\mathrm{rank}_{\mathfrak{f}} \  \ 
			\frac{ S_p^{\Sigma'}(E[p^\infty] \otimes \sigma/F_\cyc)^\vee}{\pi}.
			\end{equation}}\end{small}
	
	Recall, $S_p(E/K_\cyc)[p]$ is finite implies that $S_p(E/K_\cyc)$ is a cotorsion $\Lambda$-module. Then it is well known that the global to local maps in \eqref{exact5}  and  \eqref{389} are surjective and hence we have
	(see for example \cite[(1.3.a)]{MR2807791})
	\begin{small}{\begin{equation}\label{equ23}
			S_p^{\Sigma'}(E/K_\cyc)/ S_p(E/K_\cyc) \cong \prod_{u \mid v, v \in \Sigma'}^{} 
			\mathcal{H}_u(K_\cyc,E).
			\end{equation}}\end{small}

	Let $\de_E^{\Sigma'}(\sigma)$  denote the multiplicity of
	$\sigma$ in the $\mathcal F$ representation space  $\Big(\prod_{v \in \Sigma'}^{} 
	\mathcal{H}_v(K_\cyc,E)\Big)^\vee \otimes_{\Z_p}  \mathcal F$. For a prime $w \mid v$ of $F_\cyc$, we also consider the multiplicity of
	$\sigma$ in the $\mathcal F$ representation space 
	$ \Big(\prod_{\eta \mid w}  \HH^1(K_{\cyc,\eta},E[p^\infty])\Big)^\vee \otimes_{\Z_p}  \mathcal F$. This last multiplicity of $\sigma$ in  $ \Big(\prod_{\eta \mid w}  \HH^1(K_{\cyc,\eta},E[p^\infty])\Big)^\vee \otimes_{\Z_p}  \mathcal F$ is the same for two different $w $ and $w'$ in $F_\cyc$ lying over a prime $v$ in $F$ \cite[(5.0.b)]{MR2807791}; hence it is denoted by 
	$\de_{E,v}(\sigma)$. Let $s_v$ denote the number of primes  of $F_\cyc$ lying above $v$. Then from  \eqref{equ23}, we obtain:
	\begin{small}{\begin{equation} \label{relation}
			\lambda_E^{\Sigma'}(\sigma) =\lambda_E(\sigma)+\de_E^{\Sigma'}(\sigma)
			=\lambda_E(\sigma)+\sum_{v \in \Sigma'}^{} s_v \de_{E,v}(\sigma).
			\end{equation}}\end{small}
			
			\medskip
			
Recall an absolutely irreducible, self-dual representation $\sigma$ of $\De$ over $\mathcal F$ is said to be orthogonal (respectively symplectic) if the non-degenerate $\De$-invariant pairing $B: W_\sigma \times W_\sigma \to \mathcal F$ is symmetric (respectively skew-symmetric), where  $W_\sigma$ denotes the underlying representation space for $\sigma$.   The following relation between $\lambda_E(\sigma)$ and  $s_p(E,\sigma)$ was observed by Greenberg.
	\begin{prop} \label{prop1}
		\cite[Proposition $4.1$]{ICM}
		Let $E$ be an elliptic curve over $K$.
		If $S_p(E/K_\cyc)$ is $\Lambda$-cotorsion, then 
		for an irreducible orthogonal representation $\sigma$ of $\De$, we have
		\begin{small}{\begin{equation} \label{equ22}
				\lambda_E(\sigma) \equiv s_p(E,\sigma) \pmod{2}.  \qed 
				\end{equation}}\end{small}
	\end{prop}
	
	\medskip
	
	Note that $s_v$ in \eqref{relation} is odd, as $p$ is odd. Then using \eqref{equ22} and \eqref{relation}, we deduce the following corollary:
	\begin{cor}\label{colmarch}
		Let $E$ be good, ordinary at primes of $F$ dividing $p$. Assume that either $E[p]$ is an irreducible  $G_K$ module or  that $\Sigma'$ is non-empty. We further assume that $S_p(E/K_\cyc)[p]$ {{}is} finite.  Then for an irreducible orthogonal representation $\sigma$ of $\De$, we get 
		\begin{small}{\begin{equation} \label{relationmarch}
				s_p(E,\sigma)   \equiv \lambda_E^{\Sigma'}(\sigma)
				+\sum_{v \in \Sigma'}^{}  \de_{E,v} (\sigma)    \pmod  2. \qed 
				\end{equation}}\end{small}
	\end{cor}
	\begin{definition}\label{inner-product}
		Given a finite group $G$, we consider two finite dimensional representations $\alpha$ and $\beta$ over a field $L$ with $\alpha$ absolutely irreducible.  Let $V_\alpha$ and 
		$V_\beta$ be the corresponding representation spaces. Then define 
		$\langle \beta,\alpha \rangle = \langle \beta,\alpha \rangle_G :=
		\dim_{L}\  \mathrm{Hom}(V_\alpha, V_\beta)$.
		In other words, $\langle \beta,\alpha \rangle$ is the largest non-negative integer $m$ such that $V_\beta$
		contains a $G$-invariant subspace isomorphic to $V_\alpha^m$. Note that if $G$ is a quotient of $\mathcal{G}$, then  $\langle \beta,\alpha \rangle_G = \langle \beta,\alpha \rangle_\mathcal{G}$.
	\end{definition}
	\noindent \textbf{Determination of} $\de_{E,v}(\sigma)$: As we have noted earlier, given a prime $v$ in $F$, $\delta_{E_i,v}(\sigma)$ does not depend on the choice of $w$ in $F_\cyc$ or $\eta$  in $K_\cyc$. Hence by abuse of notation, when there is no cause of confusion, we will simply denote $F_{\cyc,w}$, $K_{\cyc,\eta}$ and $\Delta_\eta =\mathrm{Gal}(K_{\cyc,\eta}/F_{\cyc,w})$ respectively by $F_{\cyc,v}$, $K_{\cyc,v}$ and $\Delta_v$.  Recall $\sigma_v:= \sigma \mid_{\De_v}$ is the restriction of $\sigma$ to the decomposition subgroup of $\De$ at $v$.
	Let $\mathrm{Irr}_\mathcal{F}(\De_v)$ be the set of all $\mathcal{F}$ irreducible representations of  $\De_v$.
	By going to a finite extension of $\mathcal F$ if necessary, we may assume each $\chi \in \mathrm{Irr}_\mathcal{F}(\De_v)$ is absolutely irreducible. We view $\rho_{E,v}:=\rho_E \mid_{G_{F_{\cyc,v}}}$ as the representation  $\rho_{E,v}: G_{F_{\cyc,v}} \to V_p(E) \otimes_{\Q_p} \mathcal{F}$. Then we have the following  expression for $\de_{E,v}(\sigma)$ (\cite[(5.1.a), \S 5.2]{MR2807791}).
	\begin{small}{\begin{equation}\label{local}
			\de_{E,v}(\sigma)=\sum_{\chi \in \mathrm{Irr}_\mathcal{F}(\De_v)}
			\langle \sigma_v, \chi \rangle \langle \rho_{E,v}, \chi \rangle.
			\end{equation}}\end{small}
	\begin{cor}\label{colfeb27}
		Let $E$ be as in the setting of Corollary \ref{colmarch}. Then for an irreducible orthogonal representation $\sigma$ of $\De$, we get 
		\begin{small}{\begin{equation} \label{relatio26feb}
				s_p(E,\sigma)   \equiv \lambda_E^{\Sigma'}(\sigma)
				+\sum_{v \in \Sigma'}^{} \sum_{\chi \in \mathrm{Irr}_\mathcal{F}(\De_v)}
				\langle \sigma_v, \chi \rangle \langle \rho_{E,v}, \chi \rangle   \pmod  2. \qed 
				\end{equation}}\end{small}
	\end{cor}

	\begin{remark}\label{reductiontolocalstpes}
		Assume that we have two elliptic curves $E_1$ and $E_2$ which are congruent mod $p$ and they fit into the setting of Corollary \ref{colfeb27}. Then to compare the parity of $s_p(E_1,\sigma)$  with $s_p(E_2,\sigma)$, by the same corollary, it reduces to  compare (i) the Iwasawa invariants $\lambda_{E_1}^{\Sigma'}(\sigma)$ with $\lambda_{E_2}^{\Sigma'}(\sigma)$ for a suitably chosen $\Sigma'$ and then (ii) compare the  parity of  $ \langle \sigma_v, \chi \rangle\langle \rho_{E_1,v}, \chi \rangle$ with $\langle \sigma_v, \chi \rangle \langle \rho_{E_2,v}, \chi \rangle$ for each $v  $ in that particularly chosen set $\Sigma'$, with $\chi$ as above.  We will do this in the next section.
	\end{remark}
	\par  
	We now study  $\langle \rho_{E,v}, \chi \rangle$, according to the reduction type of $E$ at $v$. 
	Set $\omega_v:=\omega_p|_{G_{F_{\cyc,v}}}$, where $\omega_p:G_\Q \to \Z_p^\times$ is the $p$-adic cyclotomic character. Then $\omega_v$ is a power of the Teichm\"uller character and the order of $\omega_v$ {{}divides} $p-1.$ 
	
	\medskip
	
	\par If $E$ has good reduction at $v$, 
	then the action of $G_{F_{\cyc,v}}$ on $V_p(E)$
	is unramified and after extending $\mathcal F$ if necessary, the representation space $V_p(E) \otimes_{\Q_p}\mathcal{F}$
	is the direct sum of two one-dimensional subspaces on which $G_{F_{\cyc,v}}$ acts by unramified
	characters $\varphi_{E, v},\psi_{E,v}$ (say) of orders dividing $p^2-1$. By Weil pairing, $\varphi_{E,v} \psi_{E,v}=\omega_v$. 
	\begin{small}{\begin{equation}\label{3891}
			\langle \rho_{E,v}, \chi \rangle =
			\begin{cases} 
			2 & \text{if }  \chi =\varphi_{E,v}=\psi_{E,v} \\
			1 & \text{if } \chi \in \{\varphi_{E,v},  \psi_{E,v}\} \text { and } \varphi_{E,v} \neq \psi_{E,v} \\
			0 &  \text{otherwise } 
			\end{cases}
			\end{equation}}\end{small}
	\par Let $E$ be multiplicative  at $v$ and $\varkappa_v$ is the unique unramified quadratic 
	character of $G_{F_{\cyc,v}}$. By \cite[Prop. 2.12]{ddt}
	\begin{small}{\begin{equation} \label{varkappa}  
			\rho_{E,v} \sim 
			\begin{pmatrix}
			\omega_v & * \\
			0 & 1
			\end{pmatrix}, \text{ if } E \text{ splits at } v \ \text{ and }
			\rho_{E,v} \sim 
			\begin{pmatrix}
			\omega_v & * \\
			0 & 1
			\end{pmatrix}
			\otimes \varkappa_v, \text{ if } E \text{ is non-split at } v.
			\end{equation}}\end{small}
	\begin{small}{\begin{equation}\label{3892}
			\text{If } E \text{ is split at } v, \   \langle \rho_{E,v}, \chi \rangle =
			\begin{cases} 
			1 & \text{if } \chi= \omega_v \\
			0 &  \text{otherwise.} 
			\end{cases} \quad \ \
			\text{If } E \text{ is non-split at } v, \     \quad ~\langle \rho_{E,v}, \chi \rangle =
			\begin{cases} 
			1 & \text{if } \chi= \omega_v\varkappa_v \\
			0 &  \text{otherwise.} 
			\end{cases}
			\end{equation}}\end{small}
	\par If $E$ has additive but potentially multiplicative reduction at $v$, then $E$ is split multiplicative over a ramified quadratic extension of $F_{\cyc,v}$. Also $V_pE$ has a unique one dimensional $\rho_{E,v}$ invariant subspace on which  $\rho_{E,v}$  acts via $\omega_v\vartheta_{E,v} $ (say), where $\vartheta_{E,v} $ is a ramified quadratic character of $G_{F_{\cyc,v}}$ associated with the extension $F_v(\sqrt{-c_6})/F_v$, where $c_6$ is coming from the Weierstrass equation of $E$ \cite[Page 442-444]{sil}. Then\begin{small}{\begin{equation} \label{pmr}
			\rho_{E,v} \sim 
			\begin{pmatrix}
			\omega_v & * \\
			0 & 1
			\end{pmatrix} \otimes \vartheta_{E,v}.
			\end{equation}}\end{small}
	\begin{small}{\begin{equation}\label{3893}
			\text{In this case, \quad \ \ ~ \quad \ \ ~}\langle \rho_{E,v}, \chi \rangle =
			\begin{cases} 
			1 & \text{if } \chi= \omega_v\vartheta_{E,v} \\
			0 &  \text{otherwise } 
			\end{cases}
			\end{equation}}\end{small}
\par If $E$ has additive but potentially good reduction at $v$, then  $\rho_{E}(G_{F_{\cyc,v}})$ is finite {{} as we are over cyclotomic extension  \cite[Page 73]{MR2807791}}. Moreover, for $p \geq 5$, $p \nmid \#\rho_{E}(G_{F_{\cyc,v}})$. If  $\rho_{E}(G_{F_{\cyc,v}})$ is abelian, then the $G_{F_{\cyc,v}}$ representation $V_p(E)\otimes {\mathcal{F}}$ is a direct sum of two 
	characters, say $\varepsilon_{E,v},\varepsilon_{E,v}^{-1}\omega_v$. If $\rho_{E}(G_{F_{\cyc,v}})$ is non-abelian, then $\rho_{E,v}$ is absolutely irreducible. Thus
	\begin{small}{	 
			\begin{equation}\label{3894}
			\langle \rho_{E,v}, \chi \rangle =
			\begin{cases} 
			2 & \text{if } \chi =\varepsilon_{E,v}=\varepsilon_{E,v}^{-1}\omega_v \text{ and } \rho_{E}(G_{F_{\cyc,v}}) \text{ is abelian}\\
			1 & \text{if } \chi \in \{\varepsilon_{E,v}, \varepsilon_{E,v}^{-1}\omega_v\}, \ \varepsilon_{E,v} \neq \varepsilon_{E,v}^{-1}\omega_v \text{ and } \rho_{E}(G_{F_{\cyc,v}}) \text{ is abelian}\\
			1 & \text{if } \chi = \rho_{E,v}  \text{ and } \rho_{E}(G_{F_{\cyc,v}}) \text{ is non-abelian}\\
			0 &  \text{otherwise.} 
			\end{cases}
			\end{equation}} \end{small}
			
			\medskip
			
We end this section with a technical lemma to be used later.  Set $S^*(E[p] \otimes \widetilde{\sigma}/F_\cyc):=\underset{n}{\varprojlim} \ S_p^{\Sigma'}
	(E[p] \otimes \widetilde{\sigma}/F_n) $ where $F_n$ varies over the cyclotomic tower $F_\cyc/F$. 
	The following lemma follows from the proof of  \cite[Lemma $5.5$]{MR3766902}. As mentioned by \cite{MR3766902}, their proof follows essentially the idea of \cite[Prop 7.1]{hv}.
	
	\medskip
	
	\begin{lemma} \label{L1}
		There is an injective map 
		\begin{small}{$S^*(E[p] \otimes \widetilde{\sigma}/F_\cyc) \hookrightarrow{}
				\Hom_{\mathbb{F}_p[[\Ga]]}\big( S_p^{\Sigma'}(E[p] \otimes \widetilde{\sigma}/F_\cyc)^\vee,
				\mathbb{F}_p[[\Ga]]  \big). \qed $}\end{small}
	\end{lemma}
	
	\medskip

	\section{Comparison of the parity of  $s_p(E_i, \sigma)$} \label{s2}
	Now let $E_1,E_2$ be two elliptic curves over $k$ such that $E_1[p] \cong E_2[p]$ as $G_k$ modules. Recall from introduction, $N^F_i=N^F_{E_i}$ denote the conductor of $E_i$ over $F$ and $\overline{N}^F_i=\overline{N}^F_{E_i}$ denote the prime-to-$p$ part of the conductor of the Galois module $E_i[p]$ over $F$ for $i=1,2$. Let $\Sigma= \Sigma(E_1,E_2, K,F)$ be a finite set of primes of $F$ containing the primes above $p$,  the infinite primes, the primes of bad reduction of both $E_1$ and $E_2$ and primes which ramify in $K/F$. For a prime $v$ of $F$, let $e_v(K/F)$ denote the ramification index of $v$ in $K/F$. We now choose and fix a  subset $\Sigma_0=\Sigma_0(E_1,E_2, K,F)$ of $\Sigma$ defined as: \begin{equation} \label{Sigma_0}
	\Sigma_0= \Sigma_0(E_1,E_2,K,F):=\{v \in \Sigma: v \mid N^F_1/\overline{N}^F_1 \,\, \text{or} \,\, v \mid N^F_2/\overline{N}^F_2\}
	\cup \{v :   v \nmid p \, \infty,  \,\, \text{and} \,\, p \mid e_v(K/F)\}.
	\end{equation}
	With this definition, we can apply the results of \S \ref{sel1} on imprimitive Selmer groups with $\Sigma'=\Sigma_0$. Throughout the article, we make the following hypothesis {\bf (H)}.
	
	\par Hypothesis {\bf (H):} All of (H1) to (H4) hold  simultaneously.
	\begin{enumerate} 
		\item[ (H1)]  
		$E_1$ and  $E_2$ {{}have} good reduction at primes of $k$ lying above $p$.
		\item [(H2)] $E_1$ has ordinary reduction at primes of $k$ lying above $p$.
		\item[ (H3)] $E_1[p](K)=0.$ 
		\item[ (H4)] $S_p(E_1/K_\cyc)[p]$ is finite.
	\end{enumerate}
	\begin{remark}
		It is known (see \cite[Remark $1.2$]{MR3503694}), the assumptions (H2), (H3) and (H4) hold for $E_1$ if and only if the same hold for $E_2$. Also (H4) actually implies (H2), by a result of P. Schneider.
	\end{remark}
	
	As explained in Remark \ref{reductiontolocalstpes},  to study $s_p(E_1, \sigma)- s_p(E_2, \sigma) \pmod 2$, we will first compare $\lambda^{\Sigma_0}_{E_1}(\sigma)$ with $\lambda^{\Sigma_0}_{E_2}(\sigma)$ and  then further calculate  $\delta_{E_1, v}(\sigma)- \delta_{E_2, v}(\sigma) \pmod2,$ for each $v \in \Sigma_0$. 
	\subsection{Comparison of the parity of the multiplicities $\lambda^{\Sigma_0}_{E_i}(\sigma)$}\label{s2.1}
	The multiplication by $p$ map on $E_i$, via Kummer theory, induces  a natural map $ S_p^{\Sigma_0}(E_i[p] \otimes \widetilde{\sigma}/F_\cyc) \lra S_p^{\Sigma_0}(E_i[p^\infty] \otimes \sigma/F_\cyc) [\pi]$. We compare the parity of $\lambda^{\Sigma_0}_{E_1}(\sigma)$ with $\lambda^{\Sigma_0}_{E_2}(\sigma)$ by studying the kernel and the cokernel of this map keeping in mind that $E_1[p] \cong E_2[p]$ as Galois modules. We begin with the following Lemma:
	\begin{lemma}\label{lem00}
		Let us fix a compatible choice of primes $v,u,w,\eta$ respectively in $F,K,F_\cyc, K_\cyc$ such that 
		{{} $\eta\mid u \mid v$ and $\eta \mid w \mid v$}. Let $E \in \{E_1, E_2\}$. Assume $ v \in \Sigma \setminus \Sigma_0$ such that $v \nmid p \infty$. Then the natural map $\gamma_\eta$, induced from the multiplication by $p$ map on $E$,
	\begin{small}{\begin{equation} \label{equ21}
				\gamma_\eta:\HH^1(K_{\cyc,\eta}, E[p]) \to \HH^1(K_{\cyc,\eta}, E[p^\infty])[\pi],		\end{equation}}\end{small}
is injective. \qed 
	\end{lemma} 
		{\it Proof:} If  $E/K$ has additive reduction at $u$, then by the  proof of  \cite[Lemma $2.2$]{MR3629245},	$\HH^1(K_{\cyc,\eta},E[p^\infty])=\HH^1(K_{\cyc,\eta},E[p]) =0$  and thus Ker$(\gamma_\eta)$ in \eqref{equ21} is $=0$. On the other hand, if $u$ as above, is prime of good or multiplicative reduction of $E$ over $K$, then we claim that $u \nmid N_E^K/\overline{N}_E^K$. Assume the claim at the moment. Then using \cite[Lemma $4.1.2$]{EPW}, we deduce $E[p^\infty](K_{\cyc,\eta})$
		is divisible. By inflation-restriction sequence, we know Ker$(\gamma_\eta) \cong \HH^0(K_{\cyc,\eta},E[p^\infty])/{\pi}$ and 
		hence  the Ker$(\ga_\eta)$ in \eqref{equ21} is trivial. Thus it reduces to establish the claim that $u \nmid N_E^K/\overline{N}_E^K$ to complete the proof of the lemma. 
		
		If $E$ is  good at $v$ over $F$, then it remains good over at a prime dividing $v$ in any extension of $F$; in particular, $u \nmid N_E^K$. 	
		\par Next assume that $E/F$ has multiplicative reduction at $v$; then the same is true for $E/K$ at $u$; in particular $v || N_E^F$, $u || N_E^K$ and \begin{small}{$
				\overline{\rho}_E|_{I_v} \sim 
				\begin{pmatrix}
				1 &  *\\
				0 & 1
				\end{pmatrix}
				$}\end{small}  \cite[Proposition 2.12]{ddt}. Thus, the image of $\overline{\rho}_E|_{I_v}$ has order $p$ or $1$, depending on $*$. Recall, by our hypothesis, $v \nmid N_E^F/\overline{N}_E^F$ and hence $v \mid \mid \overline{N}_E^F$.
		If $u \mid N_E^K/\overline{N}_E^K$, 
		then we have $u \nmid \overline{N}_E^K$.
		Hence, 
		$\overline{\rho}_E|_{G_K}$ is unramified at $u$.  So  $\# \, \overline{\rho}_E (I_v)  = \# \, \overline{\rho}_E(I_v/I_u)=1$ as $p \nmid e_v(K/F).$
		Therefore, $v \nmid \overline{N}_E^F$ which is a contradiction.

		\par	Now assume that $E/F$ has additive reduction at $v$ i.e. $v^2 \mid N_E^F$. Then $v \nmid N_E^F/\overline{N}_E^F$ implies that $v^2 \mid \overline{N}_E^F$
		and so $\overline{\rho}_E|_{G_F}$ is ramified at $v$. 
		If $E/K$ becomes good at $u$, we have $u \nmid N_E^K$. Thus it reduced to consider the case where $E/K$ is multiplicative at $u$ i.e. $u \mid \mid N_E^K$.
		If $\overline{\rho}_E|_{G_K}$ is unramified at $u$, then again using $p \nmid e_v(K/F)$,
		we obtain that $\# \, \overline{\rho}_E (I_v)  = \# \, \overline{\rho}_E(I_v/I_u)=1$,
		contradicting $\overline{\rho}_E|_{G_F}$ is ramified at $v$.
		Therefore,  $\overline{\rho}_E|_{G_K}$ is ramified at $u$
		and so $u \mid \overline{N}_E^K$. As $u \mid \mid N_E^K$, $u \nmid N_E^K/{\overline{N}_E^K}$. This completes the proof of the Lemma. \qed

	\medskip
	
	\begin{prop} \label{lem1}
		Let $E\in \{E_1, E_2\}$ and $v$, in the setting of Lemma \ref{lem00}, be a  prime in $\Sigma \setminus \Sigma_0$ such that $v \nmid p \infty$. Then for any prime $w$  in $F_\cyc$ with $w\mid v$, the natural map $\ga_w$ induced by Kummer theory,
		\begin{small}{\begin{equation}\label{192inj}
				\ga_w: \HH^1(F_{\cyc,w}, E[p] \otimes \widetilde{\sigma}) \to 
				\HH^1(F_{\cyc,w}, E[p^\infty] \otimes \sigma)[\pi],
				\end{equation}}\end{small}
 is injective.  \qed 
 \end{prop}
	
	\begin{proof}
		Recall $\eta$ is a prime of $K_\cyc$ with $\eta \mid w \mid v$. By our choice of $v$ in this proposition, it follows from Lemma \ref{lem00} that $\ga_\eta$ in \eqref{equ21} is injective.
		\par Note that as $K \cap F_\cyc=F$, we can identify $\De \cong\mathrm{Gal}(K_\cyc/F_\cyc)$ and the decomposition subgroup $\De_\eta \cong\Gal(K_{\cyc,\eta}/F_{\cyc,w})$. As  $v\nmid p$ and by our hypothesis $p \nmid e(K/F)$, we deduce $p \nmid \# \De_\eta$.
		
		Tensoring equation \eqref{equ21} with $L_{\sigma}$, we obtain that 
		the map $\HH^1(K_{\cyc,\eta}, E[p]\otimes \widetilde{\sigma})
		\to \HH^1(K_{\cyc,\eta}, E[p^\infty]\otimes \sigma)[\pi]$ is injective. Now using the inflation-restriction sequence, as $p \nmid \De_\eta$ it follows that $\ga_w: \HH^1(F_{\cyc,w}, E[p]\otimes \widetilde{\sigma})
		\to \HH^1(F_{\cyc,w}, E[p^\infty]\otimes \sigma)[\pi]$ is injective as well. This completes the proof of the proposition.
	\end{proof}
	
	\medskip
	
	\noindent Now, we can prove the main result of this subsection:
	\begin{theorem}\label{thm2}
		Let $p$ be an odd prime and $k$ be a number field such that all primes of $k$ lying above $p$ are ramified in $k(\mu_p)/k$. Let $E_1, E_2$ be two elliptic curves over $k$ satisfying the hypothesis {(H)} and $E_1[p] \cong E_2[p]$ as $G_k$-modules.
		Then $\lambda_{E_1}^{\Sigma_0}(\sigma)=\lambda_{E_2}^{\Sigma_0}(\sigma)$.
	\end{theorem}
\begin{proof}
		We begin by considering the following commutative diagram induced by multiplication by $\pi$ on $E_i[p^\infty] \otimes \sigma$, for $i=1,2$:
		\begin{small}{\begin{equation}  \label{cd}
				\setlength{\arraycolsep}{1pt}
				\begin{array}{*{9}c}
				0 &\longrightarrow & S_p^{\Sigma_0}(E_i[p]\otimes \widetilde{\sigma}/F_\cyc) 
				& \longrightarrow & \HH^1(F_\Sigma/F_\cyc, E_i[p]\otimes \widetilde{\sigma}) 
				& \overset{\psi^i_{\Sigma_0}}{\longrightarrow} & \prod_{v \in \Sigma \setminus \Sigma_0}
				\mathcal{H}_v(F_\cyc,E_i[p]\otimes \widetilde{\sigma}) \\
				& & \Big\downarrow f_i & & \Big\downarrow g_i & & \Big\downarrow h_i & & \\
				0 &\longrightarrow & S_p^{\Sigma_0}(E_i[p^\infty] \otimes \sigma/F_\cyc)[\pi]
				& \longrightarrow & \HH^1(F_\Sigma/F_\cyc, E_i[p^\infty] \otimes \sigma)[\pi]
				& \overset{\phi_i[\pi]}{\longrightarrow} & \prod_{v \in \Sigma \setminus \Sigma_0}
				\mathcal{H}_v(F_\cyc, E_i[p^\infty] \otimes \sigma)[\pi].
				\end{array}
				\end{equation}}\end{small}
		Since $E_i[p^\infty] \otimes \sigma$ is a divisible $O$-module, 
		we deduce the  middle vertical map $g_i$ is surjective. 
		By assumption (H3) and  Nakayama's lemma, $E_i(K_\cyc)[p]=0$  and 
		hence the map $g_i$ is an isomorphism.
		
		Next, we will show that $\psi^i_{\Sigma_0}$ is surjective. We have the following exact sequence using Poitou-Tate duality:
		\begin{small}{\[
				0 \rightarrow S_p^{\Sigma_0}(E_i[p]\otimes \widetilde{\sigma}/F_\cyc) 
				\longrightarrow \HH^1(F_\Sigma/F_\cyc, E_i[p]\otimes \widetilde{\sigma}) 
				\overset{\psi^i_{\Sigma_0}}{\longrightarrow} \prod_{v \in \Sigma \setminus \Sigma_0}
				\mathcal{H}_v(F_\cyc,E_i[p]\otimes \widetilde{\sigma}) 
				\longrightarrow S^*(E_i[p] \otimes \widetilde{\sigma}/F_\cyc)^\vee 
				\]}\end{small}
		Thus, by Lemma \ref{L1}, to show the surjectivity of  $\psi^i_{\Sigma_0}$, it suffices to show $S_p^{\Sigma_0}(E_i[p]\otimes \widetilde{\sigma}/F_\cyc)$ is finite. 
		Now using Proposition \ref{lem1}, $\ker(h_i) \cong \prod_{v \mid p}^{} \prod_{w \mid v}^{}
		\HH^0(F_{\cyc,w}, \widetilde{E}_{i,v}[p^\infty] \otimes \sigma)/\pi$.
		As the residue field of $K_{\cyc,w}$ for $w \mid p$ is finite, we obtain $\HH^0(F_{\cyc,w}, \widetilde{E}_{i,v}[p^\infty] \otimes \sigma)$ is finite and consequently, ker$(h_i)$ is finite. By  assumption (H4),  $S_p(E_i[p^\infty]/K)[p]$ is finite. As $g_i$ is an isomorphism and ker$(h_i)$ is finite, applying snake lemma in diagram \eqref{cd},  it follows that  $S_p^{\Sigma_0}(E_i[p] \otimes \tilde{\sigma}/F_\cyc)$ is finite. Thus $\psi^i_{\Sigma_0}$ is surjective.
		
		Let $v \mid p$ be a prime of $F$. As $p$ is odd and the action of the inertia subgroup of $G_{F_v}$ on $\ker(E_i[p] \to \overline{E}_{i,v}[p])$ is via the Teichm\"uller character, by the choice of $k$ we conclude this action is nontrivial. Hence $\overline{E}_{i,v}[p]$ can be characterized as the maximal unramified quotient of $E_i[p]$ for $i=1,2$. Now, since $E_1[p] \cong E_2[p]$ as $G_k$-modules, we have $\widetilde{E}_{1,v}[p] \cong \widetilde{E}_{2,v}[p]$ as $G_{F_v}$-modules for every prime $v \mid p$ of $F$ (see the preceding paragraph of \cite[Theorem $2.2$]{MR3503694}). It then follows that
		$\widetilde{E}_{1,v}[p] \otimes \widetilde{\sigma} \cong \widetilde{E}_{2,v}[p] \otimes \widetilde{\sigma}$
		as $G_{F_v}$-modules. Then, from the definition of $S_p^{\Sigma_0}(E_i[p]\otimes \widetilde{\sigma}/F_\cyc)$, we obtain that 
		\begin{small}{\begin{equation} \label{ism2}
				S_p^{\Sigma_0}(E_1[p]\otimes \widetilde{\sigma}/F_\cyc) \cong 
				S_p^{\Sigma_0}(E_2[p] \otimes \widetilde{\sigma}/F_\cyc).
				\end{equation}}\end{small} In particular, \begin{small}{$\mathrm{rank}_{\mathfrak{f}} \  S_p^{\Sigma_0}(E_1[p]\otimes \widetilde{\sigma}/F_\cyc) = \mathrm{rank}_{\mathfrak{f}} \ S_p^{\Sigma_0}(E_2[p]\otimes \widetilde{\sigma}/F_\cyc)$}\end{small}. Moreover, notice that
		$$\mathrm{rank}_{\mathfrak{f}}  \  \  \HH^0(F_{\cyc,w}, \widetilde{E}_{i,v}[p^\infty] \otimes \sigma)/\pi= \mathrm{rank}_{\mathfrak{f}}  \  \  
		\HH^0(F_{\cyc,w}, (\widetilde{E}_{i,v}[p^\infty] \otimes \sigma)[\pi]) =\mathrm{rank}_{\mathfrak{f}}  \  \  
		\HH^0(F_{\cyc,w}, \widetilde{E}_{i,v}[p] \otimes \tilde{\sigma})$$
		Hence,  $\mathrm{rank}_{\mathfrak{f}} \  \text{ker}(h_1) = \mathrm{rank}_{\mathfrak{f}} \ \text{ker}(h_2)$.

		Now we know  $\psi^i_{\Sigma_0}$  is surjective and $g$ is an isomorphism. Then again applying snake lemma in diagram \eqref{cd}, it is immediate that   
		$\mathrm{rank}_{\mathfrak{f}} \ S_p^{\Sigma_0}(E_1[p^\infty] \otimes \sigma/F_\cyc)^\vee/{\pi} = \mathrm{rank}_{\mathfrak{f}} \ S_p^{\Sigma_0}(E_2[p^\infty] \otimes \sigma/F_\cyc)^\vee/\pi$. This completes the proof of the theorem.
	\end{proof}
	\begin{remark}
		Recall, a prime $v \mid p$ of $F$ is said to be a non-anomalous prime for $E/K$ 
		if $|\widetilde{E}_v(k_u)|$ is not divisible by $p$, for every prime  $u \mid v$  of $K$, where $k_u$ is the residue field at $u$.   As before, assume the hypothesis {\bf (H)}. In addition, assume either (i) $v$ is non-anomalous for $E/K$ or (ii) $(\widetilde{E}_v[p^\infty] \otimes \sigma)^{G_{F_{\cyc,w}}}$ is divisible,
		for primes $w \mid v$ of $F_\cyc$.  Then from the proof of Theorem \ref{thm2}, we can deduce that
		\begin{small}{\begin{equation} \label{ism3}
				S_p^{\Sigma_0}(E[p]\otimes \widetilde{\sigma}/F_\cyc) \cong
				S_p^{\Sigma_0}(E[p^\infty] \otimes \sigma/F_\cyc)[\pi].
				\end{equation}}\end{small}
	\end{remark}	
	Using Theorem \ref{thm2} and Corollary \ref{colmarch},  we deduce the following corollary:
	\begin{cor}\label{col8}
		Let us keep the hypotheses and setting of Theorem \ref{thm2}. Then for an irreducible orthogonal representation $\sigma$ of $\De$, we get 
		\begin{small}{\begin{equation} \label{relation8}
				s_p(E_1,\sigma) +\sum_{v \in \Sigma_0}^{}  \de_{E_1,v} (\sigma) \equiv s_p(E_2,\sigma) +\sum_{v \in \Sigma_0}^{}  \de_{E_2,v} (\sigma)    \pmod  2. \qed 
				\end{equation}}\end{small}
	\end{cor}
	Further by Corollary \ref{colfeb27} and Corollary \ref{col8}, the comparison of parity of $s_p(E_1,\sigma)$ and $s_p(E_2,\sigma)$ reduces to comparing  the parity of $\langle \sigma_v, \chi \rangle\langle \rho_{E_i,v}, \chi \rangle$ with $\chi \in  \mathrm{Irr}_\mathcal{F}(\De_v)$, for each $v \in \Sigma_0$.
	
	\subsection{Comparison of the parity of $\de_{E_i,v}(\sigma) $}  \label{sub1} 
	
	In this subsection, we  compare the parity of the local Iwasawa invariants $\de_{E_1,v}(\sigma) $ with $\de_{E_2,v}(\sigma) $ via the study of the quantity $\langle \sigma_v, \chi \rangle(\langle \rho_{E_1,v}, \chi \rangle - \langle \rho_{E_2,v}, \chi \rangle) \pmod 2$, according to the reduction type of $E_1$ and $E_2$ at each $v \in \Sigma_0$. 
	\begin{lemma}[good - good] \label{lem4}
		Let $v \nmid p$ be a prime of $F$ in $\Sigma_0$ such that $E_1$ and $E_2$ are good  at $v$. Then $\de_{E_1,v}(\sigma) \equiv \de_{E_2,v}(\sigma) \pmod{2}$.
	\end{lemma}
	\begin{proof}
		$E_1$ and $E_2$ are good  at $v$. Then using equation  \eqref{3891} in the relation  \eqref{local},  we only need to compute $\langle \rho_{E_1,v}, \chi \rangle  - \langle \rho_{E_2,v}, \chi \rangle \pmod 2 $  for $\chi \in \{\varphi_{E_1,v}, \psi_{E_1,v}, \varphi_{E_2,v}, \psi_{E_2,v} \}$. We know $V_p(E_i) \otimes_{\Q_p}\mathcal{F} \cong \mathcal{F}(\varphi_{E_i,v})\oplus  \mathcal{F}(\psi_{E_i,v})$ as $G_{F_{\cyc,v}}$ modules. Also $E_1[p] \cong E_2[p]$ as $G_\Q$-modules and $\varphi_{E_i,v}, \psi_{E_i,v}$ are unramified characters of order prime to $p$. Thus we deduce $\{\varphi_{E_1,v}, \psi_{E_1,v}\} = \{\varphi_{E_2,v}, \psi_{E_2,v}\}$. Now using \eqref{local}, $\de_{E_1,v}(\sigma) \equiv \de_{E_2,v}(\sigma) \pmod{2}$.
	\end{proof}
	\begin{lemma}[split - split] \label{lem8}
		Let $v \in \Sigma_0$ be a prime of split multiplicative reduction for $E_1$ and $E_2$. Then 
		$\de_{E_1,v}(\sigma) \equiv \de_{E_2,v}(\sigma) \pmod{2}$.
	\end{lemma}
	\begin{proof}
		In this case, using equation  \eqref{3892} in the relation  \eqref{local}, $\de_{E_i,v}(\sigma) \equiv \langle \sigma_v, \omega_v \rangle \pmod{2}$, which is the same for both $E_1$ and $E_2$. Thus $\de_{E_1,v}(\sigma) \equiv \de_{E_2,v}(\sigma) \pmod{2}$.
	\end{proof}
	\begin{lemma} [non-split - non-split] \label{lem9}
		Let $E_1$ and $E_2$ both  have non-split multiplicative reduction at a $v \in \Sigma_0$. Then $\de_{E_1,v}(\sigma)\equiv \de_{E_2,v}(\sigma) \pmod{2}$.
	\end{lemma}
	\begin{proof}
		Recall $\varkappa_v $ is defined in \eqref{varkappa}. Again in this case, using equation  \eqref{3892} in the relation  \eqref{local}, $\de_{E_i,v}(\sigma) \equiv \langle \sigma_v, \omega_v\varkappa_v \rangle \pmod{2}$, which is the same for both $E_1$ and $E_2$. Thus $\de_{E_1,v}(\sigma) \equiv \de_{E_2,v}(\sigma) \pmod{2}$.
	\end{proof}
	\begin{lemma}[good - split]\label{lem11}
		Let $v \in \Sigma_0$. Assume $E_1$ has good reduction at $v$
		and $E_2$ has split multiplicative reduction at $v$, then 
		$\de_{E_1,v}(\sigma) \equiv \de_{E_2,v}(\sigma) + \langle \sigma_v, 1_v \rangle \pmod{2}$.
	\end{lemma}
	\begin{proof}
		In this case, $\rho_{E_1,v} \sim$ \begin{small}{$\begin{pmatrix}
				\varphi_{E_1,v} & 0 \\
				0 & \psi_{E_1,v}
				\end{pmatrix}$}\end{small} and $ \rho_{E_2,v} \sim$ 
		\begin{small}{$\begin{pmatrix}
				\omega_v & * \\
				0 & 1
				\end{pmatrix}$}\end{small}.
		Now  $E_1[p] \cong E_2[p]$ implies that $\{\overline{\varphi}_{E_1,v}, \overline{\psi}_{E_1,v}\} = \{\overline{\omega}_v, \overline{1}_v\} =\{\omega_v, 1_v\}$.
		As $\varphi_{E_1,v},
		\psi_{E_1,v}$ and  $\omega_v$ have all orders prime to $p$, we see that 
		$\{\varphi_{E_1,v}, \psi_{E_1,v}\} = \{\omega_v, 1_v\}$.
		Therefore, using  \eqref{3891} we obtain that
		$\de_{E_1,v}(\sigma)\equiv \langle \sigma_v, \omega_v \rangle 
		+ \langle \sigma_v, 1_v \rangle \pmod{2}$. On the other hand, by  \eqref{3892},
		$\de_{E_2,v}(\sigma)\equiv \langle \sigma_v, \omega_v\rangle \pmod{2}$. 
		Hence, the result follows.
	\end{proof}
	\begin{lemma} [good - non-split] \label{lem12}
		Let $v \in \Sigma_0$. Assume $E_1$ has good reduction at $v$
		and $E_2$ has non-split multiplicative reduction at $v$, then 
		$\de_{E_1,v}(\sigma) \equiv \de_{E_2,v}(\sigma) + \langle \sigma_v, \varkappa_v \rangle \pmod{2}$.
	\end{lemma}
	\begin{proof}
		The proof is similar to the good-split case in Lemma \ref{lem11} and is omitted.
	\end{proof}
	\begin{lemma} [split - non-split] \label{lem10}
		Let  $v \in  \Sigma_0$. Assume  $E_1$ has split multiplicative reduction at $v$ and $E_2$ has non-split multiplicative reduction $v$. Then $\de_{E_1,v}(\sigma) \equiv \de_{E_2,v}(\sigma) + \langle \sigma_v, 1_v \rangle + \langle \sigma_v, \varkappa_v \rangle \pmod{2}$. 
	\end{lemma}
	\begin{proof}
		By the proof of Lemma \ref{lem8}, $\de_{E_1,v}(\sigma) \equiv \langle \sigma_v, \omega_v \rangle \pmod 2$. Similarly, by Lemma \ref{lem9}, $\de_{E_2,v}(\sigma) \equiv \langle \sigma_v, \omega_v\varkappa_v \rangle \pmod 2$. Thus $\de_{E_1,v}(\sigma) \equiv \de_{E_2,v}(\sigma) + \langle \sigma_v, \omega_v \rangle + \langle \sigma_v, \omega_v\varkappa_v \rangle \pmod 2$. 
		\par \noindent By \cite[Lemma 2.8]{MR3629245}, the 3 conditions  (i) $E_1$ is split at $v$, (ii) $E_2$ is non-split at $v$ and (iii) $E_1[p] \cong E_2[p]$, are simultaneously satisfied only when $\omega_v = \varkappa_v$. As $\varkappa_v$ is quadratic, the lemma is now immediate. \end{proof}
\begin{lemma}\label{image} 
		If $E$ is potentially good at $v$, then $\rho_{E}(G_{F_{cyc,v}})$ and $\overline{\rho}_{E}(G_{F_{cyc,v}})$ are isomorphic. 
	\end{lemma}
\begin{proof}
		By \cite[Corollaries 2(b) \& 3]{seta}, $E$ acquires good reduction over $F_{\cyc,v}(E[p])$. Set $G:= \Gal(\overline{F}_v/F_{\cyc,v}(E[p]))$. Let $g \in G_{F_{\cyc,v}}$ be an element such that $\overline{\rho}_{E,v}(g)$ acts trivially on $E[p]$. Then, $g \in G$. Moreover, $G$ acts on $V_pE$ by the sum of two characters of order prime to $p$, after a suitable base change. Hence, $g$ acts trivially on $V_p(E)$ if and only if $\overline{\rho}_{E,v}(g)$ acts trivially on $E[p]$. This proves the lemma.
	\end{proof}
\begin{lemma}[good - additive] 
		Let $v \nmid p\infty$ be a prime of $F$. Then it is not possible that $E_1$ has good reduction while $E_2$ has additive reduction at $v$.  In particular, this is true for any $v \in \Sigma_0$. 
	\end{lemma}
	\begin{proof}
		As $E_1$ is good at $v$, $E_1[p]$ is unramified. On the other hand, by Lemma \ref{image}, $E_2[p]$
		is ramified. Since $E_1[p] \cong E_2[p]$, the lemma follows.
	\end{proof}
	Abbreviation: We write $E$ is PMR at a prime $v$ in $F$, if $E$ has additive and potentially multiplicative reduction at $v$. By $E$ is PGA  (respectively PGNA) at a prime $v$ in $F$, we mean $E$ is additive, potentially good at $v$ with $\rho_E(G_{F_{\cyc,v}})$ abelian (respectively non-abelian).
	\begin{lemma} [multiplicative - additive] \label{Lem1}
		Let  $v \in \Sigma_0$ be a multiplicative prime of  $E_1$ and additive  prime for $E_2$. Then by \cite[Lemma $2.5$]{MR3629245}, it follows that $p=3$. Moreover, the following hold.
		\begin{enumerate}
			\item[(i)]It can not happen that $E_1$ is multiplicative  at $v$  and $E_2$ is PMR  at $v$.
			\item [(ii)]Assume that $E_1$ is split (resp. non-split) multiplicative at $v$  and $E_2$ is PGA at $v$. Then $\mu_3 \subset F_v$. If $v$ is split for $E_1$ then $\de_{E_1,v}(\sigma) \equiv \de_{E_2,v}(\sigma)+\langle \sigma_v, 1_v \rangle  \pmod{2}$ and if $v$ is non-split for $E_1$ then  $\de_{E_1,v}(\sigma) \equiv \de_{E_2,v}(\sigma)+\langle \sigma_v, \varkappa_v \rangle  \pmod{2}$.
			\item [(iii)]It is not possible that $E_1$ is multiplicative  at $v$  and  $E_2$ is PGNA  at $v$,  unless $ 3 \mid \# \rho_{E_2}(G_{F_{\cyc,v}})$ and $\mu_3 \not\subset F_v$. In the later case, $\de_{E_1,v}(\sigma)  + \langle \sigma_v, \chi_v \rangle \equiv \de_{E_2,v} (\sigma) + \langle \sigma_v, \theta_{3,v} \rangle \pmod{2} $, where $\theta_{3,v}$, is the unique $2$-dimensional irreducible representation of $\mathrm{Gal}(F_{\cyc,v}(E_2[p])/F_{\cyc,v})$  which is the dihedral group $D_6$ of order $6$ \cite[Page 165-166]{MR2807791}. Here $\chi_v = \varkappa_v$ if $E_1$ is split at $v$ and $\chi_v = 1_v$ if $E_1$ is non-split at $v$. 
		\end{enumerate}
	\end{lemma}
\proof \begin{enumerate}
			\item [(i)] If $E_1$ has multiplicative reduction at $v$, then the semi-simplification of $E_1[p]$ is unramified. On the other hand, if $E_2$ has PMR at $v$ then the semi-simplification of $E_2[p]$ is ramified. Thus, we cannot have this case. 
			\item [(ii)]
			Using  \eqref{3892} and \eqref{3894}, $\de_{E_1,v}(\sigma) \equiv \langle \sigma_v, \omega_v \rangle \pmod{2}$ if $E_1$ is split at $v$, $\de_{E_1,v}(\sigma) \equiv \langle \sigma_v, \omega_v\varkappa_v \rangle \pmod{2}$ if $E_1$ is non-split at $v$ and $\de_{E_2,v}(\sigma) \equiv \langle \sigma_v, \varepsilon_{E_2,v} \rangle 
			+ \langle \sigma_v, \varepsilon_{E_2,v}^{-1}\omega_v \rangle \pmod{2}$.

			Suppose that   $3 \nmid \#\rho_{E_2}(G_{F_{\cyc,v}})$. Then by Lemma \ref{image}, $3 \nmid \#\bar{\rho}_{E_2}(G_{F_{\cyc,v}})$. Since $\overline{\rho}_{E_1}\mid_{I_{F_v}} \sim$ \begin{small}{$\begin{pmatrix}
					1 & * \\
					0 &1
					\end{pmatrix}$}\end{small},  we deduce that $E_1[p]$ is unramified and hence so is $E_2[p]$. This is a contradiction to the assumption that $E_2$ has  potentially good additive reduction (see \cite{seta}).  		
			\par Thus we are reduced to consider $3 \mid \#\rho_{E_2}(G_{F_{\cyc,v}})$. Then $p=3$ and we have  $\overline{\rho}_{E_1}\mid_{I_{F_v}} \sim$   \begin{small}{$\begin{pmatrix}
					1 & * \\
					0 &1
					\end{pmatrix}$}\end{small} with $* \neq 0$. Also, $\overline{\rho}_{E_1}\mid_{G_{F_{cyc,v}}} \sim$ \begin{small}{$\begin{pmatrix}
					\omega_v & * \\
					0 &1
					\end{pmatrix}$}\end{small}. This implies that $\overline\rho_{E_1}(G_{F_{\cyc,v}})$ is a non-abelian group if $\omega_v$ is non-trivial. Therefore  $\overline\rho_{E_1}(G_{F_{\cyc,v}})$ is also non-abelian. This is a contradiction. Therefore $\omega_v$ is trivial and hence we must have $\mu_3\subset F_v$.   Since $\sigma_v$ is self dual, $ \langle \sigma_v, \varepsilon_{E_2,v} \rangle= \langle \sigma_v, \varepsilon_{E_2,v}^{-1} \rangle $. Therefore $\de_{E_2,v}(\sigma) \equiv 0 \pmod 2$. This proves the claim.
			
			\item [(iii)] Here $\de_{E_1,v}(\sigma) \equiv \langle \sigma_v, \omega_v \rangle \pmod{2}$ if $E_1$ is split at $v$, $\de_{E_1,v}(\sigma) \equiv \langle \sigma_v, \omega_v\varkappa_v \rangle \pmod{2}$ if $E_1$ is non-split at $v$ and $\de_{E_2,v}(\sigma) \equiv \langle \sigma_v, \rho_{E_2,v} \rangle \pmod{2}$. Note that for $p=3$, $\omega_v =\varkappa_v$ if $\mu_3 \not\subset F_v$.
			
			If $3 \nmid \#\rho_{E_2}(G_{F_{\cyc,v}})$, then by the argument in part (ii) above,  $E_2[p]$ is unramified and we get a contradiction. So, we can assume $3  \mid \# \overline\rho_{E_2}(G_{F_{\cyc,v}})$. Further,    $\mu_3\subset F_{\cyc,v}$ is also not possible; otherwise, $\overline\rho_{E_1}(G_{F_{\cyc,v}})$ is abelian and  by Lemma \ref{image},  $\rho_{E_2}(G_{F_{\cyc,v}})$ is  abelian as well. 
			\par Therefore we can assume $3 \mid \#\rho_{E_2}(G_{F_{\cyc,v}})$ and $\mu_3\not\subset F_{\cyc,v}$. Then $\overline{\rho}_{E_1}\mid_{G_{F_{cyc,v}}}$ is of order 6 and by Lemma \ref{image}, $\rho_{E_2}(G_{F_{\cyc,v}}) \cong \overline\rho_{E_2}(G_{F_{\cyc,v}})$ is non-abelian. Consequently, $\rho_{E_2}(G_{F_{\cyc,v}})$ is isomorphic to $D_6$.  Let $\theta_{3,v}$ be the unique two dimensional irreducible representation of the dihedral group $\Gal(F_{cyc,v}(E_2[p])/F_{\cyc,v})$. We have that $\rho_{E_2,v}\sim \theta_{3,v}$. Thus the lemma follows.  \qed
\end{enumerate}
\begin{lemma} [additive - additive]\label{lem5}
		Let $v \in \Sigma_0$ be an additive prime for both $E_1$ and $E_2$. Then,		
		\begin{enumerate}
			\item[(i)]
			\label{lem698}
			If $E_1, E_2$ both are PMR at $v\in \Sigma_0$, then $\de_{E_1,v}(\sigma)  \equiv \de_{E_2,v}(\sigma) \pmod{2}$ if  $\vartheta_{E_1,v}= \vartheta_{E_2,v}$, otherwise, $\de_{E_1,v}(\sigma) +  \langle \sigma_v, \vartheta_{E_1,v} \rangle \equiv \de_{E_2,v}(\sigma) + \langle \sigma_v, \omega_v \vartheta_{E_1,v} \rangle \pmod{2}$.
			\item [(ii)]
			It can not happen that $E_1$ is PMR  at $v$ while $E_2$ is PGNA  at $v$, unless $p=3$,  3 divides $\# \rho_{E_2}(G_{F_{\cyc,v}})$ and $\mu_3 \not\subset F_v$. In the later case, $\de_{E_1,v}(\sigma)  + \langle \sigma_v, \omega_v\vartheta_{E_1,v} \rangle \equiv \de_{E_2,v} (\sigma) + \langle \sigma_v, \theta_{3,v} \vartheta_{E_1,v} \rangle \pmod{2} $. Moreover,  $\mathrm{Gal}(F_{\cyc,v}(E_2[p])/F_{\cyc,v})$ is the dihedral group of order 12 and it has a unique subgroup, say  $H \cong D_6$. Then $\theta_{3,v}$ is the unique $2$-dimensional irreducible representation of $H$.
			\item [(iii)]
			If $E_1$ is PMR  at $v$ and $E_2$ is PGA at $v$, then 
			$\de_{E_1,v}(\sigma) \equiv \de_{E_2,v}(\sigma)+ \langle 
			\sigma_v, \vartheta_{E_1,v} \rangle \pmod{2}$. 
			\item [(iv)]
			If $E_1, E_2$ are both PGA at $v$, 
			then $\de_{E_1,v}(\sigma) \equiv \de_{E_2,v}(\sigma) \pmod{2}$. 
			\item [(v)]  
			It is not possible that one of $E_1, E_2$ is PGA and the other one is PGNA at $v$.
			\item [(vi)] 
			Let $E_1, E_2$ be both PGNA at $v$. Then $\de_{E_1,v}(\sigma)  \equiv \de_{E_2,v}(\sigma)  \pmod{2}$.
		\end{enumerate}
\end{lemma}
\proof
\begin{enumerate}
			\item [(i)] 
			Using  (\ref{3893}),  $\de_{E_i,v}(\sigma)\equiv \langle \sigma_v, \omega_v\vartheta_{E_i,v} \rangle \pmod{2}$. As $E_1[p] \cong E_2[p]$ and  $\vartheta_{E_1,v}, \vartheta_{E_2,v}$ have orders prime to $p$, we deduce from \eqref{pmr}  that either  $\vartheta_{E_1,v}=\vartheta_{E_2,v}$ or $\vartheta_{E_1,v}=\omega_v\vartheta_{E_2,v}$ and the result follows. Note that  $\vartheta_{E_1,v}=\omega_v\vartheta_{E_2,v}$ is possible only when $\omega_v$ is quadratic. 
			\item [(ii)]
			Using  (\ref{3893}) and \eqref{3894}, $\de_{E_1,v}(\sigma)\equiv \langle \sigma_v, \omega_v\vartheta_{E_1,v} \rangle \pmod{2}$ and $\de_{E_2,v}(\sigma)\equiv \langle \sigma_v, \rho_{E_2,v} \rangle \pmod{2}$. Note that  $E_1':= E_1^{\vartheta_{E_1,v}}$ has split multiplicative reduction at $v$.  Also, since $\rho_{E_2,v}$ is absolutely irreducible, $E_2':=E_2^{\vartheta_{E_1,v}}$ is PGNA at $v$. Also, $E_1'[p] \cong E_2'[p]$.
			
			First suppose that $p \nmid \#\rho_{E_2}(G_{F_{\cyc,v}})$. Since $\vartheta_{E_1,v}$ is quadratic and $p$ is odd, it follows that $p \nmid \#\rho_{E_2'}(G_{F_{\cyc,v}})$. Now, by an argument similar to   Lemma \ref{Lem1}(iii), it is not possible that $E_1'$ is split at $v$, $E_2'$ is potentially good with $\rho_{E_2'}(G_{F_{\cyc,v}})$ is non-abelian and $E_1'[p] \cong E_2'[p]$. 
			
			Hence,  $p$ must divide $\#\rho_{E_2}(G_{F_{\cyc,v}})$. By the discussion after \eqref{3893}, this can happen only when $p=3$. Also, as explained above, $p \mid \#\rho_{E_2'}(G_{F_{\cyc,v}})$. If $\mu_3 \subset F_{v}$, then $\rho_{E_2'}(G_{F_{\cyc,v}})$ is abelian (see Lemma \ref{Lem1}(iii)), which is not possible.  So we can assume $\mu_3 \not \subset F_{v}$. 
			By Lemma \ref{image}, $3 \mid \#\overline{\rho}_{E_2'}(G_{F_{\cyc,v}})=  \#\overline{\rho}_{E_1'}(G_{F_{\cyc,v}})$. Using these facts in the expression $\overline{\rho}_{E_1'}\mid_{G_{F_{\cyc,v}}} \sim \begin{pmatrix}
			\omega_v & * \\
			0 &1
			\end{pmatrix}$, we obtain that $\#\overline{\rho}_{E_1'}({G_{F_{\cyc,v}}}) = 6= \#\overline{\rho}_{E_2'}({G_{F_{\cyc,v}}})$. Thus, by Lemma \ref{image}, $\rho_{E_2'}(G_{F_{\cyc,v}})$ is non-abelian of order $6$. It follows that $\rho_{E_2',v} \cong \theta_{3,v}$ and $\rho_{E_2,v} \cong \theta_{3,v}\vartheta_{E_1,v}$. The result follows.

			\item [(iii)]
			Using \eqref{3893} and \eqref{3894},  $\de_{E_1,v}(\sigma)\equiv \langle \sigma_v, \omega_v\vartheta_{E_1,v} \rangle \pmod{2}$ and $\de_{E_2,v}(\sigma) \equiv \langle \sigma_v, \varepsilon_{E_2,v} \rangle + \langle \sigma_v, \omega_v\varepsilon_{E_2,v}^{-1} \rangle \pmod{2}$. 
			Set $E_1':= E_1^{\vartheta_{E_1,v}}$ and $E_2':= E_2^{\vartheta_{E_1,v}}$. Then $E_1'$ is split multiplicative at $v$, $E_2'$ is PGA at $v$ and $E_1'[p] \cong E_2'[p]$.
			
			If possible, let $p \nmid \# \rho_{E_2}(G_{F_{\cyc,v}})$. Then as explained in part(ii), $p \nmid \#\rho_{E_2'}(G_{F_{\cyc,v}})$. Now, since $E_1'[p] \cong E_2'[p]$, this is not possible by the proof of Lemma \ref{Lem1}(ii). 
			
			Thus, $p$ must divide $\# \rho_{E_2}(G_{F_{\cyc,v}})$ and hence $p \mid \#\rho_{E_2'}(G_{F_{\cyc,v}})$. This can only happen at $p=3$ (see the discussion after \eqref{3893}). Further,  using  Lemma \ref{Lem1}(ii), $\mu_3 \subset F_{\cyc,v}$ and  $\de_{E_2,v}(\sigma) \equiv 0 \pmod{2}$. So we arrive at the result.

			\item [(iv)]
			Using \eqref{3894},  $\de_{E_i,v}(\sigma)\equiv \langle \sigma_v, \varepsilon_{E_i,v} \rangle + \langle \sigma_v, \omega_v\varepsilon^{-1}_{E_i,v} \rangle \pmod{2}$.
			When  $p \nmid \#\rho_{E_1}(G_{F_{\cyc,v}})\#\rho_{E_2}(G_{F_{\cyc,v}})$, the proof is similar to Lemma \ref{lem4} and is omitted. 
			
			Assume that  $p$ divides the order of $\rho_{E_i}(G_{F_{\cyc,v}})$ for either $i=1$ or $i=2$. Then we necessarily have $p=3$. Also for that $i$, $\varepsilon_{E_i,v}$  has order divisible by $3$. Then $F_{\cyc,v}$ will have a cyclic $p$-extension with ramification index divisible by $3$ and  by local class field theory, $\mu_3 \subset F_{\cyc,v}$. As a result, $\omega_v =1_v$. Since $\sigma_v$ is self-dual,  $\de_{E_2,v}(\sigma) \equiv \langle \sigma_v, \varepsilon_{E_2,v} \rangle + \langle \sigma_v, \varepsilon_{E_2,v}^{-1}  \rangle \equiv 0 \pmod{2}$. 
			\item [(v)]By Lemma \ref{image}, $\overline{\rho}_{E_1}(G_{F_{\cyc,v}})$ is abelian and $\overline{\rho}_{E_2}(G_{F_{\cyc,v}})$ is non-abelian but they are isomorphic. This is not possible. 
			\item [(vi)]  Again recall, $\rho_{E_i}(G_{F_{\cyc,v}}) \cong \overline{\rho}_{E_i}(G_{F_{\cyc,v}})$ by Lemma \ref{image}. Note that $\rho_{E_i,v}$ is absolutely irreducible. From \cite[Page 312]{ser}, the image $\rho_{E_i}(I_{F_v})$ is 
			\begin{itemize}
				\item 
				a cyclic group of order $2, 3, 4$ or $6$, if $v \nmid 6$ 
				\item 
				a cyclic group of order $2, 3, 4, 6$ or a semi-direct product $\Z/3\Z \rtimes \Z/4\Z$, if $v \mid 3$
				\item 
				a subgroup of $\mathrm{SL}_2(\mathbb{F}_3)$ of order $2, 3, 4, 6, 8$ or $24$, if $v \mid 2$.
			\end{itemize}
			\par First consider the case that  $p \nmid \#\rho(G_{F_{\cyc,v}}) $. When $p \nmid \#\rho(G_{F_{\cyc,v}}) $, the reduction mod $p$ induces a bijection between the set of irreducible representations $\rho$ of $G_{F_{\cyc,v}}$  and  the set of all irreducible representations $\overline{\rho}$ of $G_{F_{\cyc,v}}$ (see for instance \cite[P. 85-86]{MR1984740}). Using this fact together with $\overline{\rho}_{E_1,v} \cong \overline{\rho}_{E_2,v}$, we deduce that $\rho_{E_1,v} \cong \rho_{E_2,v}$.  Now from \eqref{3894},  $\de_{E_i,v}(\sigma)\equiv \langle \sigma_v, \rho_{E_i,v} \rangle \pmod{2}$. Therefore, $\de_{E_1,v}(\sigma) \equiv \de_{E_2,v}(\sigma) \pmod{2}$. 
			\par In the second case, when $p \mid \#\rho(G_{F_{\cyc,v}}) $, again recall that $p=3$ and hence $v\nmid 3$. Thus the possibilities we have that $v \mid 2$ and $\# \rho_{E_i}(I_{F_v}) \in \{3, 6, 24 \}$ or $v \nmid 6$ and $\# \rho_{E_i}(I_{F_v}) \in \{3, 6\}$. 
			
			\par Let us consider the subcase when $\# \rho_{E_i}(I_{F_v})=3$ for $i=1,2$.  As $\rho_{E_i}(I_{F_v})$ is cyclic of order $3$ and the image of $\rho_{E_i,v}$ is non-abelian, $\rho_{E_i}(G_{F_{\cyc,v}}) \cong \rho_{E_i}(I_{F_{v}}) \rtimes \Z/2\Z   \cong D_6$ (see \cite[Page 153]{MR1260960}). Since there is a unique $2$-dimensional irreducible representation of $D_6$, we deduce $\rho_{E_1,v} \cong \rho_{E_2,v}$. Therefore, $\de_{E_1,v}(\sigma) \equiv \de_{E_2,v}(\sigma) \pmod{2}$. 
			
			Next subcase is when $\# \rho_{E_i}(I_{F_v})=6$ for $i=1,2$. As  $\mathrm{SL}_2(\mathbb F_3)$ has no subgroup isomorphic to $D_6$, $\rho_{E_i}(I_{F_v}) \cong \Z/6\Z$. Once again, by using arguments of \cite[Page 153]{MR1260960}, we obtain the structure $\rho_{E_i}(G_{F_{\cyc,v}}) \cong \rho_{E_i}(I_{F_{v}}) \rtimes \Z/2\Z \cong D_{12}$. Since there is only one faithful $2$-dimensional irreducible representation of $D_{12}$, as observed from the character table of $D_{12}$, we conclude that $\rho_{E_1,v} \cong \rho_{E_2,v}$. Hence, $\de_{E_1,v}(\sigma) \equiv \de_{E_2,v}(\sigma) \pmod{2}$. 
			
			Finally, we are left to deal with the subcase  where $p=3$ and $ \rho_{E_i}(I_{F_v}) \cong \mathrm{SL}_2(\mathbb F_3)$ of order  24. Then $v \mid 2$. Also notice that, as Weil pairing is non-degenerate, $\mu_3\subset F_v$ if and only if $\rho_{E_i}(G_{F_{\cyc,v}}) \subset \mathrm{SL}_2(\mathbb F_3)$.  Hence if $\mu_3 \subset F_v$, then $\rho_{E_i,v}$ is self-dual and $\rho_{E_i}(G_{F_{\cyc,v}}) \cong \SL_2(\mathbb{F}_3)$ for $i=1,2$.   From the character table,  there is only one  $2$-dimensional irreducible representation self dual representation  of $\SL_2(\mathbb{F}_3)$. Hence, $\rho_{E_i,v}$ is uniquely determined and we deduce $\rho_{E_1,v} \cong \rho_{E_2,v}$.  
			On the other hand, if $\mu_3 \not \subset F_v$, then $\rho_{E_i}(G_{F_{\cyc,v}})  \cong \GL_2(\mathbb{F}_3)$ for $i=1,2$. 
			From the character table, $\GL_2(\mathbb{F}_3)$ has three  $2$-dimensional irreducible representations and further, all three of these 2 dimensional representations have distinct traces $\pmod 3$.  Since $\overline{\rho}_{E_1}|_{G_{F_{\cyc,v}}} \cong \overline{\rho}_{E_2}|_{G_{F_{\cyc,v}}}$, we deduce $\rho_{E_1,v} \cong \rho_{E_2,v}$. So $\de_{E_1,v}(\sigma) \equiv \de_{E_2,v}(\sigma) \pmod{2}$. \qed 
		\end{enumerate} 
	\subsection{The main theorem on the algebraic side}\label{mainspei}
	Let $B$ be the set of primes of $F$, such that  both $E_1$ and $E_2$ have additive reduction and at least one of them has PMR. At each $v\in B$, we define a quadratic character $\vartheta_v$:  if $E_1$ (respectively $E_2$) has PMR at $v$ and $E_2$ (respectively $E_1$) has potentially good reduction then put $\vartheta_v:= \vartheta_{E_1,v}$ (respectively $\vartheta_{E_2,v}$).  When both $E_1$ and $E_2$ have PMR at $v$, put $\vartheta_v:=\vartheta_{E_1,v}$. Note, in this last case by Lemma  \ref{lem5}(i),  either $\vartheta_{E_1,v}=\vartheta_{E_2,v}$ or  $\vartheta_{E_1,v}=\omega_v\vartheta_{E_2,v}$ with $\omega_v$ quadratic. The table below summarises the cases where $\de_{E_1,v}(\sigma)$ and $\de_{E_2,v}(\sigma)$ can have different parity. {{}In the following table, we use the notation $a \equiv_2 b$ to denote $a \equiv b \pmod 2$}. {\small
		\begin{table}[htbp]
			\centering
			\begin{tabular}{ |l|l|l| }
				\hline
				${\bf p, \,\,\, v \nmid p}$ & {\bf \ \ \ \ \ \ \ \ \ \ \  \ Reduction type at $v$} & {\bf \ \ \ \ \ \ \ \ \ \ \ \ \ \ \ Parity} \\ 
				\hline
				\multirow{2}{*}{$p \geq 3$} &
				\multirow{2}{*}{$E_1$ is good and  $E_2$ is split} & 
				\multirow{2}{*}{$\de_{E_1,v}(\sigma) \equiv_2 \de_{E_2,v}(\sigma)+\langle \sigma_v, 1_v \rangle$} \\ & {} & {}\\
				\hline
				\multirow{2}{*}{$p \geq 3$} &
				\multirow{2}{*}{$E_1$ is good and  $E_2$ is non-split} & 
				\multirow{2}{*}{$\de_{E_1,v}(\sigma) \equiv_2 \de_{E_2,v}(\sigma)+\langle \sigma_v, \varkappa_v \rangle$} \\ & {} & {}\\
				\hline
				\multirow{2}{*}{$p \geq 3$} &
				\multirow{2}{*}{$E_1$ is split and $E_2$ is non-split} & 
				\multirow{2}{*}{$\de_{E_1,v}(\sigma) \equiv_2 \de_{E_2,v}(\sigma)+\langle \sigma_v, 1_v \rangle + \langle \sigma_v, \varkappa_v \rangle $} \\ & {} & {}\\
				\hline
				\multirow{2}{*}{$p \geq 3$} &
				\multirow{2}{*}{$E_1, E_2$ are PMR and $\vartheta_{E_1,v} \neq \vartheta_{E_2,v}$} & 
				\multirow{2}{*}{$\de_{E_1,v}(\sigma) +\langle \sigma_v, \vartheta_{E_1,v} \rangle \equiv_2 \de_{E_2,v}(\sigma)+ \langle \sigma_v, \omega_v\vartheta_{E_1,v}\rangle $} \\ & {} & {}\\
				\hline
				\multirow{2}{*}{$p \geq 3$} &
				\multirow{2}{*}{$E_1$ is PMR and  $E_2$ is PGA}& 
				\multirow{2}{*}{$\de_{E_1,v}(\sigma) \equiv_2 \de_{E_2,v}(\sigma)+\langle \sigma_v, \vartheta_{v} \rangle $} \\ & {} & {}\\
				\hline
				\multirow{2}{*}{$p=3, \,\, \mu_3 \subset F_v$} &
				\multirow{2}{*}{$E_1$ is split and  $E_2$ is PGA}& 
				\multirow{2}{*}{$\de_{E_1,v}(\sigma) \equiv_2 \de_{E_2,v}(\sigma)+\langle \sigma_v, 1_v \rangle$} \\ & {} & {}\\
				\hline
				\multirow{2}{*}{$p=3, \,\, \mu_3 \subset F_v$} &
				\multirow{2}{*}{$E_1$ is non-split and  $E_2$ is PGA}& 
				\multirow{2}{*}{$\de_{E_1,v}(\sigma) \equiv_2 \de_{E_2,v}(\sigma)+\langle \sigma_v, \varkappa_v \rangle$} \\ & {} & {}\\
				\hline
				\multirow{2}{*}{$p=3, \,\,\mu_3 \not\subset F_v$} &
				\multirow{2}{*}{$E_1$ is split and  $E_2$ is PGNA}& 
				\multirow{2}{*}{$\de_{E_1,v}(\sigma)+\langle \sigma_v, \varkappa_v \rangle \equiv_2 \de_{E_2,v}(\sigma)+ \langle \sigma_v, \theta_{3,v} \rangle$} \\ & {} & {}\\
				\hline
				\multirow{2}{*}{$p=3, \,\,\mu_3 \not\subset F_v$} &
				\multirow{2}{*}{$E_1$ is non-split and  $E_2$ is PGNA}& 
				\multirow{2}{*}{$\de_{E_1,v}(\sigma)+\langle \sigma_v, 1_v \rangle \equiv_2 \de_{E_2,v}(\sigma)+ \langle \sigma_v, \theta_{3,v} \rangle $} \\ & {} & {}\\
				\hline
				\multirow{2}{*}{$p=3, \,\,\mu_3 \not\subset F_v$} &
				\multirow{2}{*}{$E_1$ is PMR and  $E_2$ is PGNA} & 
				\multirow{2}{*}{$\de_{E_1,v}(\sigma) + \langle \sigma_v, \omega_v\vartheta_{E_1,v} \rangle \equiv_2 \de_{E_2,v}(\sigma)+\langle \sigma_v, \theta_{3,v}\vartheta_{E_1,v} \rangle$} \\ & {} & {} \\
				\hline 
			\end{tabular}
			\caption{Parity of $\de_{E_i,v}(\sigma)$}
			\label{table2}
		\end{table}
	}
	\begin{definition}\label{exp-prime}
		Let $S_{i}, N_{i}$ respectively  denote the set of primes $v \in \Sigma_0$ such that the elliptic curve $E_i/F$  has  split multiplicative, respectively non-split multiplicative reduction at $v$. Let $\Sigma_0'$ denote the set the primes in $\Sigma_0$ where $\sigma$ is ramified.  
		\par $W = \{ v\in \Sigma_0' : \mu_p\not\subset F_v,  \text{both $E_1, E_2$ are PMR   at $v$  and $\vartheta_{E_1,v}\neq \vartheta_{E_2,v}$}\}$.
		\par $X=  \{ v \in \Sigma_0': \text{one of $E_1,E_2$ is PMR  at $v$ and the other is PGA at $v$} \}.$
		\par $Y_3=  \{v \in \Sigma_0: p=3, \mu_3\not\subset F_v,  \text{one of $E_1,E_2$ is  multiplicative  at $v$ and the other  is PGNA at $v$} \}$.
		\par $Z_3=  \{v \in \Sigma_0' : p=3, \mu_3\not\subset F_v,  \text{one of $E_1,E_2$ is  PMR  at $v$ and the other  is  PGNA at $v$} \}. $
	\end{definition}
	Using Corollary \ref{col8} and combining  Lemmas \ref{lem4} to \ref{Lem1}, we obtain the following theorem:
	\begin{theorem} \label{thm1}
		Let $p$ be an odd prime and $k$ be a number field such that all primes of $k$ lying above $p$ are ramified in $k(\mu_p)/k$. Let $E_1, E_2$ be two elliptic curves over $k$ with $E_1[p] \cong E_2[p]$ as $G_k$-modules. Let $\sigma$ be an absolutely irreducible self-dual orthogonal representation of $\De$. Assume the hypothesis {\bf (H)}. Then 
		\begin{small}{\begin{eqnarray*}
					&&s_p(E_1, \sigma) + \sum_{v \in S_1} \langle \sigma_v, 1_v \rangle 
					+
					\sum_{v \in N_1}^{} \langle \sigma_v, \varkappa_v \rangle  + \sum_{v\in W} \big( \langle \sigma_v, \vartheta_v \rangle +  \langle \sigma_v,  \omega_v\vartheta_v \rangle \big)  + \sum_{v \in Z_3} \big( \langle \sigma_v, \omega_v\vartheta_{v} \rangle +  \langle \sigma_v, \vartheta_{v}\theta_{3,v} \rangle \big)\\
					& & \equiv s_p(E_2,\sigma) + \sum_{v \in S_2} \langle \sigma_v, 1_v \rangle 
					+
					\sum_{v \in N_2}^{} \langle \sigma_v, \varkappa_v \rangle +
					\sum_{v \in X} \langle \sigma_v, \vartheta_{v} \rangle
					+ \sum_{v \in Y_3}  \big(   \langle \sigma_v,   1_v \rangle  + \langle \sigma_v,   \varkappa_v \rangle +  \langle \sigma_v,   \theta_{3,v}\rangle\big)
					\pmod{2}. \qed 
		\end{eqnarray*}}\end{small}
	\end{theorem}
	\begin{cor}\label{lascorr}
		We keep the hypotheses and setting of Theorem \ref{thm1}. In addition, whenever $v$ is a prime of   additive reduction for at least one of $E_1, E_2$,  we assume that $\mu_p \subset F_v$. Then
		\begin{small}{\begin{eqnarray*}
					s_p(E_1, \sigma) + \sum_{v \in S_1} \langle \sigma_v, 1_v \rangle 
					+
					\sum_{v \in N_1}^{} \langle \sigma_v, \varkappa_v \rangle 
					\equiv  s_p(E_2,\sigma) + \sum_{v \in S_2}^{} \langle \sigma_v, 1_v \rangle 
					+
					\sum_{v \in N_2} \langle \sigma_v, \varkappa_v \rangle 
					+
					\sum_{v \in X} \langle \sigma_v, \vartheta_v \rangle
					\pmod{2}. \qed 
		\end{eqnarray*}}\end{small}
	\end{cor}
	\begin{remark} \label{equal}
		From the definitions in \eqref{varkappa} and \eqref{pmr}, it is clear that $\varkappa_v$ and $\vartheta_v$ can be considered as characters of $G_{F_{v}}$ as well. Now, let $E_1$ is split mult. at $v$ and $E_2$ is PGNA at $v$. Then, necessarily $p=3$ and $\Gal(F_{\cyc,v}(E_1[p])/F_{\cyc,v}) \cong D_6$ (see Lemma \ref{Lem1}(iii)). Further, $\overline\rho_{E_1}|_{G_{F_v}} \sim {\small \begin{pmatrix}
			\chi_1 & * \\
			0 &\chi_2
			\end{pmatrix}} $  with $\{ \chi_1,\chi_2 \}=\{ \omega_v, 1_v \}$ and $\omega_v\neq 1_v$. Hence,  $\Gal(F_v(E_1[p])/F_v)$ is also  the non-abelian of order $6$, i.e., $\Gal(F_v(E_1[p])/F_v) \cong \Gal(F_{\cyc,v}(E_1[p])/F_{\cyc,v})$ and  $F_v(E_1[p]) \cap F_{\cyc,v}=F_v$. In fact, $\Gal(F_{\cyc,v}(E_1[p])/F_v) \cong \Gal(F_{\cyc,v}(E_1[p])/F_{\cyc,v}) \times \Gal(F_{\cyc,v}/F_v)$. Thus we may as well  consider  $\theta_{3,v}$ to be the unique irreducible representation of  degree two of $\Gal(F_v(E_1[p])/F_v)\cong \Gal(F_v(E_2[p])/F_v)$. 
		
		\par Let $\tau_v$ be a sub-representation of $\sigma |_{G_{F_v}}$ such that $\tau_v|_{G_{F_{\cyc,v}}}$ is equivalent to $\theta_{3,v}|_{G_{F_{\cyc,v}}}$. Then $\tau_v|_{G_{F_v}}\sim \theta_{3,v}|_{G_{F_v}}\otimes \alpha$  for a  character $\alpha$ of $\Gal(F_{\cyc,v}/F_v)$ of $p$-power order. Since $\sigma |_{G_{F_v}}$ is self-dual and $p$ is odd,  $\langle \sigma |_{G_{F_v}},\theta_{3,v}|_{G_{F_v}}\otimes \alpha,\rangle= \langle \sigma |_{G_{F_v}}, \theta_{3,v}|_{G_{F_v}}\otimes \alpha^{-1},\rangle$. This shows that $\langle \sigma |_{G_{F_v}},\theta_{3,v}|_{G_{F_v}}\rangle =\langle \sigma |_{G_{F_{\cyc,v}}},\theta_{3,v}|_{G_{F_{\cyc,v}}}\rangle \pmod 2$.  A similar argument works for every Artin representations $\tau_v$ in the set $C=\{\varkappa_v,1_v, \vartheta_v, \vartheta_v\theta_{3,v},\theta_{3,v} \}$. Thus the multiplicities $<\sigma_v, \dagger>, \dagger \in C$ in Theorem \ref{thm1}  can be considered over $G_{F_v}$ as well.
		\par \noindent This observation will be used in \S \ref{s7} to compare these multiplicities  with the local root numbers over $F_v$.
	\end{remark}%
	
	\section{Relation with the root numbers}\label{s7}
	\subsection{Basic results on root number}\label{s5}
	Let $k, F, K$ and $E \in \{E_1, E_2\}$ be as before. Recall $\sigma: \De=\begin{small}{\mathrm{Gal}(K/F)  \rightarrow \mathrm{GL}_n(\mathcal F)}\end{small}$ is a  self-dual Artin representation.  Also, $O$ is the ring of integer of $\mathcal F$ with uniformizer $\pi$. In view of Remark \ref{equal}, we now use the notation $\sigma_v:=\sigma\mid_{G_{F_v}}$. 
	For a prime $v$ of $F$, let (i) $\rho'_{E/F_v}$, (ii) $\rho_{E/F_v}$ and (iii) $\rho_{E,v}= \rho_E\mid_{G_{F_{\cyc,v}}}$ respectively denote (i) the representations of the Weil-Deligne group $ \mathcal W'(F_v)$ of $F_v$ associated to $E$, (ii) the Weil group $\mathcal W(F_v)$ of $F_v$ associated to $E$ \cite[\S 2, \S 3]{MR1260960} and (iii) the restriction of the $p$-adic Galois representation $\rho_E: G_k \rightarrow \mathrm{Aut}(V_pE)$ (see \S \ref{sel1})  at $G_{F_{\cyc,v}}$.  
	\par We will denote by $V_E$ the representation space for $\rho'_{E/F_v}$. Set $V_E^\sigma:=V_E \otimes \sigma$. Let $\varepsilon(V^\sigma_E)$ denote the local epsilon factor associated to $V_E^\sigma$ at $v$ \cite{MR0349635}. The local root number, $W(E/F_v,\sigma_v)$, at $v$ is then given by \cite{MR1387669}
	\begin{equation} \label{localroot}
	W(E/F_v, \sigma_v)=\mathrm{sign} (\varepsilon(V_E^\sigma))=\frac{\varepsilon(V_E^\sigma)}{|\varepsilon(V_E^\sigma)|} .
	\end{equation}
	Since $\sigma$ is self dual, $W(E/F_v, \sigma_v)\in \{\pm 1\}$ \cite[Equ. (3.1)]{MR1387669}. The global root number  $W(E/F,\sigma)$, is a product of all local root numbers i.e. $W(E/F,\sigma) :=\underset{\text{all prime } v}\prod  W(E/F_v,\sigma_v)$ \cite{MR1387669}. 
	
	Note that the local  Artin reciprocity map carries $F_v^\times$  isomorphically onto 
	$\mathcal W(F_v)^{\mathrm{ab}}$ inside Gal$(F_v^{ab}/F_v)$. 
	In particular, if $\phi$ is a character of $G_{F_v}$, then $\phi(-1)$ is well defined and again by the local reciprocity law, 
	$\phi(-1)=1$ if $\phi$ is unramified (see for example, \cite[\S 1]{MR1260960}).
	Let $\omega(v)$ denote the unramified character of $\mathcal{W}(F_v)$ sending a Frobenius element to $q_v$, the order of the residue field of $F_v$. We recall the following results of Rohrlich which will be used to compare the local root numbers of $E_1$ and $E_2$.
	\begin{theorem}\label{root}
		Let $E/F$ be an elliptic curve, $\sigma$ be a self-dual, irreducible representation  of $\De$. 
		\begin{enumerate}
			\item [(i)] \cite[Theorem 2(i)]{MR1387669} If $v$ is an archimedean prime of $F$, then  $W(E/F_v,\sigma_v)=(-1)^{\dim \sigma}$.
			\item[(ii)]\cite[Prop. 8(i)]{MR1387669} If  $E$ has good reduction at a finite prime $v$ of $F$, then $W(E/F_v,\sigma_v)=\det \sigma_v(-1)$. 
			\item [(iii)] \cite[Theorem 2(ii)]{MR1387669}
			If $\mathrm{ord}_v(j_E)<0$, then $W(E/F_v,\sigma_v)=\det\sigma_v(-1) \phi_v(-1)^{\dim\sigma} (-1)^{\langle \sigma_v, \phi_v \rangle}$. Here 
			$\phi_v=1_v$ if $v$ is split multiplicative, $\phi_v=\varkappa_v$ is the unramified quadratic character if $v$ is non-split multiplicative, and $\phi_v =\vartheta_{E,v}$ (see \eqref{pmr}) is ramified quadratic if E is PMR at $v$. 
			\item [(iv)] 
			Let $v$ be a prime of additive, potentially good reduction of $E$. 	
			\begin{enumerate}
				\item [(a)]
				\cite[Prop. 2]{roh3} If $E$ is PGA at $v$, then $\rho_{E/F_v} \otimes \omega(v)^{1/2} \cong \chi_{E,v} \oplus \chi_{E,v}^{-1}$, for some character $\chi_v$ of $F_v^\times$. 
				In this case \cite[(1.6)]{roh3}, $W(E/F_v,\sigma_v)=\det\sigma_v(-1) \chi_{E,v}(-1)^{\dim\sigma}$. 
				\item [(b)] Let $E$ be PGNA at $v$. If $v \nmid 6$, then by the proof of \cite[Theorem 2(iii), page 332]{MR1387669},  $\rho_{E_2/F_v} \otimes \omega(v)^{1/2}=\Ind_{L/F_v}\phi$, where $L/F_v$ is an unramified quadratic extension and $\phi$ is a character of $L^\times$. Write $L=F_v(u)$, for some $u \in O_L^\times$. Further by \cite[Theorem 1]{MR1387669}, $W(E/F_v,\sigma_v)=\det\sigma_v(-1)  \phi(u)^{\dim\sigma}  (-1)^{\langle \sigma_v,1_v \rangle+\langle \sigma_v, \varkappa_v \rangle+\langle \sigma_v, \theta_{e,v} \rangle}$, where $e:= \#\rho_E(I_{F_v})=\#\overline{\rho}_E(I_{F_v})$ and $\theta_{e,v}$ is {{}the faithful 2-dimensional}
representation of \begin{small}{$\mathrm{Gal}(F_v(E[p])/F_v) \cong D_{2e}$.}\end{small}    \qed 
			\end{enumerate}
		\end{enumerate}
	\end{theorem}
	
	\begin{remark} \label{pga}
		(1) Recall, $\rho'_{E/F_v}$ is a pair $(\rho_{E/F_v},N)$, where $N$ is a nilpotent endomorphism on $V_E$ satisfying $\rho_{E/F_v}(g)N\rho_{E/F_v}(g)^{-1}=\omega(v)(g) N$ for all $g \in \mathcal{W}(F_v)$ \cite[\S 3]{MR1260960}. When $E$ is additive and potentially good at $v \mid \ell$, with $\ell \neq p$ then $N=0$ and $\rho'_{E/F_v}=\rho_{E/F_v}$ \cite[Prop. \S 14]{MR1260960}. 
		\par In the case (iv)(a) of Theorem \ref{root}, we have $\rho_{E,v} \cong \varepsilon_{E,v} \oplus \omega_v \varepsilon_{E,v}^{-1}$ (cf. \S \ref{sel1}).  As any real power of $\omega(v)$ is unramified, we see from the definition of a representation of Weil-Deligne group \cite[Prop. \S 4]{MR1260960} that $\varepsilon_{E,v}(-1)=\chi_{E,v}(-1)$.
		\par (2) In Theorem \ref{root} (iv)(b), if $v \mid 2$ or $v \mid 3$, we see from the proof of \cite[Theorem 2(iii)]{MR1387669} that as long as $\rho_{E/F_v}$ is an induced representation {{} from an unramified quadratic extension}, we can continue to use the formula in Theorem \ref{root} (iv)(b).
	\end{remark}
	\subsection{Comparison of the  root numbers}\label{rootnumbereisigma}
	Note  that $\rho'_{E_i/F_v} \otimes \omega(v)^{1/2}$ is symplectic \cite[\S 16]{MR1260960}.  Let $a(\rho'_{E/F_v})$ denote the conductor {{}exponent} of $\rho'_{E/F_v}$ (see \cite[\S 10]{MR1260960}). At first, we prove the following:
	\begin{lemma} \label{frac}
		Let $E_1, E_2$ be two elliptic curves with $E_1[p] \cong E_2[p]$ as $G_k$-modules. Let $v \mid \ell$ with $\ell \neq p$ be a finite prime of $F$ such that $a(\rho'_{E_1/F_v} \otimes \sigma)=a(\rho'_{E_2/F_v} \otimes \sigma)$. Then 
		\begin{small}{\begin{equation*} 
				\frac{W(E_1/F_v, \sigma_v)}{W(E_2/F_v, \sigma_v)} = \frac{\det\big(-\Frob_{\cyc,v} \mid  (V^\sigma_{E_2})^{I_{F_v}} \big)}{\det\big(-\Frob_{\cyc,v} \mid  (V^\sigma_{E_1})^{I_{F_v}} \big)} \pmod{\pi}. 
				\end{equation*}}\end{small} 
	\end{lemma} 
	
		\proof \par Let $E \in \{E_1,E_2 \}$. For $v \mid \ell$ with $\ell \neq p$, recall the modified local constant, $\varepsilon_0(V^\sigma_E)$, defined by Deligne \cite[Equ. (5.1) \& Theorem 6.5(a)]{MR0349635}:
		\begin{equation} \label{modified}
		\varepsilon_0(V^\sigma_E)=\varepsilon(V^\sigma_E) \det\big(-\Frob_v \mid  {(V^\sigma_E)}^{I_{F_v}} \big) \in O^\times
		\end{equation}
		Further from the proof of \cite[Theorem $6.5(c)$]{MR0349635}, it can be shown that $\varepsilon_0(V^\sigma_E)  \pmod \pi$ depends on the residual representation $E[p] \otimes \widetilde \sigma$. In particular, $\varepsilon_0(V^\sigma_{E_1})= \varepsilon_0(V^\sigma_{E_2}) \pmod \pi$. 
		\par   Now, as $\rho'_{E_i/F_v} \otimes \omega(v)^{1/2}$ is symplectic 
		and $\sigma$ is self-dual, we can apply \cite[Lemma \S 12]{MR1260960} to deduce $|\varepsilon(V^\sigma_{E_i})| = \ell^{m_i}$. 
		Further using $a(\rho'_{E_1/F_v} \otimes \sigma)=a(\rho'_{E_2/F_v} \otimes \sigma)$,  we also deduce from \cite[Lemma \S 12]{MR1260960},  $m_1=m_2$ i.e. $|\varepsilon(V^\sigma_{E_1})|=|\varepsilon(V^\sigma_{E_2})|$. 
		\par  Recall, $W(E_i/F_v, \sigma_v) \in \{\pm1\}$. As $|\varepsilon(V^\sigma_{E_i})| = \ell^{m_i}$, using \eqref{localroot} and \eqref{modified} we deduce that $\det\big(-\Frob_v \mid  (V^\sigma_{E_i})^{I_{F_v}} \big)$ is a $p$-adic unit. Hence, $\det \big(-\Frob_v \mid  (V^\sigma_{E_i})^{I_{F_v}} \big) \pmod \pi$ is an element of $\overline{\mathbb{F}}_p^\times$. Again since $\Gal(F_{\cyc,v}/F_v)$ is a pro-$p$ group, we obtain that 
		\begin{small}{\begin{equation} \label{cyclo}
		\det\big(-\Frob_v \mid  (V^\sigma_{E_i})^{I_{F_v}} \big) = \det\big(-\Frob_{\cyc,v} \mid  (V^\sigma_{E_i})^{I_{F_v}} \big) \pmod \pi.
		\end{equation}}\end{small}
		Using \eqref{cyclo} in \eqref{localroot} and \eqref{modified}, we obtain the result. \qed
	\par With the notation of Theorem \ref{thm1}, put 
	\begin{small}{$m_i:=\sum_{v \in S_i} \langle \sigma_v, 1_v \rangle  + \sum_{v \in N_i}^{} \langle \sigma_v, \varkappa_v \rangle$ for $i=1,2$, and $T:=\sum_{v\in W} \big( \langle \sigma_v, \vartheta_v \rangle +  \langle \sigma_v,  \omega_v\vartheta_v \rangle \big) + \sum_{v \in X} \langle \sigma_v, \vartheta_{v} \rangle + \sum_{v \in Y_3}  \big(\langle \sigma_v,   \varkappa_v \rangle  + \langle \sigma_v,   1_v \rangle  +  \langle \sigma_v,   \theta_{3,v}\rangle\big) +	\sum_{v \in Z_3} \big( \langle \sigma_v, \vartheta_{v} \rangle +  \langle \sigma_v, \vartheta_{v}\theta_{3,v} \rangle \big) $.}\end{small} Then we have the following comparison of the root numbers of two congruent elliptic curves:

	\begin{theorem} \label{propo4} 	
		Let $E_1, E_2$ be two elliptic curves over $k$ with good reductions at all primes of $k$ lying above an odd prime $p  $ and $E_1[p] \cong E_2[p]$ as $G_k$-modules. Let $\sigma
		$ be an irreducible, self-dual, orthogonal representation of $\De$. Then \begin{small}{ \[
				\frac{W(E_1/F,\sigma)}{W(E_2/F,\sigma)}=(-1)^{m_1-m_2+T}.
				\]}\end{small}
	\end{theorem}
	\begin{proof}
		We will compare $W(E_i/F_v, \sigma_v)$ for $i=1,2$,  as $v$ varies. By Theorem \ref{root}(i), we only need to consider finite primes. Further, if $E_1$ and $ E_2$ have good reduction at  $v$, then by Theorem \ref{root}(ii), $W(E_1/F_v, \sigma_v) =W(E_2/F_v, \sigma_v)$. In particular, this holds for all $v \mid p$. By Theorem \ref{root}(iii), the same conclusion holds if $E_1$ and  $ E_2$ are both split-multiplicative  or both  non-split multiplicative  at $v$. 
		\begin{enumerate}
			\item[(i)]
			(good - multiplicative).
			Suppose that $E_1$ has good reduction at $v$ and $E_2$ has multiplicative reduction at $v$. By Theorem \ref{root}, case (ii) and (iii), we deduce
			$
			\frac{W(E_1/F_v,\sigma_v)}{W(E_2/F_v,\sigma_v)}=
			\begin{cases} 
			(-1)^{\langle \sigma_v,1_v \rangle}, & \text{if } v \ \ \text{is split} \\
			(-1)^{\langle  \sigma_v, \varkappa_v \rangle}, &   \text{if } v \ \ \text{is non-split} .
			\end{cases}
			$
			
			\item[(ii)]
			(split - non-split).
			Assume that $E_1$ is split multiplicative at $v$ and $E_2$ is non-split multiplicative at $v$. By Theorem \ref{root}, we get
			$\frac{W(E_1/F_v,\sigma_v)}{W(E_2/F_v,\sigma_v)}= (-1)^{\langle \sigma_v,1_v \rangle + \langle \sigma_v, \varkappa_v \rangle}.$
			
			\item [(iii)]
			(multiplicative - additive)
			
			(a) Let $E_1$ be multiplicative at $v$ and $E_2$ be PGA at $v$. From the proof of Lemma \ref{Lem1}(ii), $p=3$,  $3 \mid \# \rho_{E_2}(G_{F_{\cyc,v}})$ and $\mu_3\subset F_v$. In this case, $\rho_{E_2,v} \cong \varepsilon_{E_2,v} \oplus \varepsilon_{E_2,v}^{-1}$. Since $\rho_{E_1}|_{I_{F_v}} \sim \begin{small}{\begin{pmatrix}
				1 & * \\
				0 & 1
				\end{pmatrix}}\end{small}$ and $\overline{\rho}_{E_1,v} \cong \overline{\rho}_{E_2,v}$, we deduce $\overline{\varepsilon}_{E_2,v} $ is unramified and thus $\overline{\varepsilon}_{E_2,v}(-1) =\varepsilon_{E_2,v}(-1)=1$.  
			Then by  Remark \ref{pga} and the cases (iii) \& (iv)(a) of Theorem \ref{root},  $\frac{W(E_1/F_v,\sigma_v)}{W(E_2/F_v,\sigma_v)}= (-1)^{\langle \sigma_v, 1_v \rangle}$  if $v$ is split and $= (-1)^{\langle \sigma_v, \varkappa_v \rangle} $ if $v$ is non-split.

			\vskip 2mm
			
			(b) Next we consider  $E_1$ is multiplicative at $v$ and $E_2$ is PGNA at $v$. From the proof of Lemma \ref{Lem1}(iii), this can only happen when $p=3$, $3 \mid \# \rho_{E_2}(G_{F_{\cyc,v}})$, $\mu_3 \not\subset F_v$ and $e=\# \rho_{E_2}|_{I_{F_v}}=3$. Moreover, $\rho_{E_2,v} \cong \theta_{3,v}$ is the unique $2$-dim irreducible representation of $\Gal(F_{\cyc,v}(E_2[p])/F_{\cyc,v}) \cong D_6$.
			Using Theorem \ref{root}(iv)(b),	 $\rho_{E_2/F_v} \otimes \omega(v)^{1/2}=\Ind_{L/F_v}\phi$ with $L/F_v$ is unramified quadratic 
			and $W(E_2/F_v, \sigma_v)=\det \sigma_v(-1) (\phi(u))^{\dim\sigma}(-1)^{\langle \sigma_v, 1_v \rangle + \langle \sigma_v, \varkappa_v \rangle + \langle \sigma_v, \theta_{3,v} \rangle}.$ 
			\par Note that the value of $\phi(u)$ is independent of $\sigma_v$. By \cite[Lemma 5.4]{MR3629245},  {\small$\frac{W(E_1/F_v,1_v)}{W(E_2/F_v,1_v)}$} $=1$ if  $E_1$ is spilt at $v$ and $-1$ if $E_1$ is non-split at $v$. Therefore, we deduce that  $\phi(u)$ appearing in the expression of $W(E_2/F_v, \sigma_v)$ above is equal to $1$. Now using Theorem \ref{root}(iii), from the above discussion {\small $\frac{W(E_1/F_v,\sigma_v)}{W(E_2/F_v, \sigma_v)}= (-1)^{\langle \sigma_v, \varkappa_v \rangle+\langle \sigma_v, \theta_{3,v} \rangle}$} if $E_1$ is split at $v$. Similarly, in the non-split case {\small $\frac{W(E_1/F_v,\sigma_v)}{W(E_2/F_v,\sigma_v)}=  (-1)^{\langle \sigma_v, 	1_v \rangle+\langle \sigma_v, \theta_{3,v} \rangle}$}. 
			\item[(iv)](additive - additive).
			\begin{enumerate}	
				\item[(a)] If both the curves are PMR at $v$, then  by the proof of Lemma \ref{lem5}(i), either  $\vartheta_{E_1,v}=\vartheta_{E_2,v}$  in which case $W(E_1/F_v,\sigma_v)=W(E_2/F_v,\sigma_v)$. Otherwise,  $\vartheta_{E_1,v}=\omega_v\vartheta_{E_2,v}$ and then, $\frac{W(E_1/F_v,\sigma_v)}{W(E_2/F_v,\sigma_v)}=  (-1)^{\langle \sigma_v, \vartheta_{E_1,v} \rangle+\langle \sigma_v, \omega_v\vartheta_{E_1,v} \rangle}$.
				\item[(b)] Next, assume $E_1$ is PMR at $v$ and $E_2$ is PGA at $v$. By Remark \ref{pga}(1) and the cases (iii) and (iv)(a) of Theorem \ref{root}, we can express $\frac{W(E_1/F_v,\sigma_v)}{W(E_2/F_v,\sigma_v)} =  \vartheta_{v}(-1)^{\dim\sigma} (-1)^{\langle \sigma_v, \vartheta_{E_1,v} \rangle}\varepsilon_{E_2,v}(-1)^{\dim\sigma}$. If $p \nmid \# \rho_{E_2}(G_{F_{\cyc,v}})$, then using the proof of Lemma \ref{lem5}(iii),  $\vartheta_{E_1,v} \in \{\varepsilon_{E_2,v},\varepsilon_{E_2,v}^{-1}\omega_v\}$. As $\omega_v$ is unramified,  $\frac{W(E_1/F_v,\sigma_v)}{W(E_2/F_v,\sigma_v)}= (-1)^{\langle \sigma_v, \vartheta_{E_1,v} \rangle}.$
				\par \noindent Now  if $p \mid \# \rho_{E_2}(G_{F_{\cyc,v}})$, then as explained before, $p=3, \mu_3 \subset F_v$ and  $\omega_v =1_v$. Thus \begin{small}{$\rho_{E_1,v} \sim \begin{pmatrix} \vartheta_{E_1,v} & * \\ 0 & \vartheta_{E_1,v}\end{pmatrix}$}\end{small} and $\rho_{E_2,v} \cong \varepsilon_{E_2,v} \oplus \varepsilon_{E_2,v}^{-1}$. As $\overline{\rho}_{E_1,v} \cong \overline{\rho}_{E_2,v}$,  we conclude that $\vartheta_{E_1,v}(-1)=\varepsilon_{E_2,v}(-1)$. Then using Remark \ref{pga} and Theorem \ref{root} (iii) \& (iv)(a), $\frac{W(E_1/F_v,\sigma_v)}{W(E_2/F_v,\sigma_v)}= (-1)^{\langle \sigma_v, \vartheta_{E_1,v} \rangle}$.
				\item[(c)]Let $E_1$ be PMR at $v$ and $E_2$ be PGNA at $v$. Then $p =3, \mu_3 \not\subset F_v $ and $3 \mid  \#\rho_{E_2}(G_{F_{\cyc,v}})$ (see Lemma \ref{lem5}(ii)).  Then $E_1':=E_1^{\vartheta_{E_1,v}}$ is split at $v$. Recall, $\rho_{E_2,v} \cong \theta_{6,v}$,  a certain  $2$-dim. irreducible representation of the dihedral group of order $12$. Also, $E_2':=E_2^{\vartheta_{E_1,v}}$ is PGNA at $v$. Note that $E_1'[3] \cong E_2'[3]$. Then by a similar argument, as in the proof of Lemma \ref{lem5}(ii), we get $\# \overline{\rho}_{E_2'}|_{I_{F_v}}=\#\rho_{E_2'}|_{I_{F_v}}=3$. Therefore, $\rho_{E_2',v} \cong \theta_{3,v}$, that is, $\rho_{E_2,v} \otimes \vartheta_{E_1,v} \cong \theta_{6,v} \otimes \vartheta_{E_1,v} \cong \theta_{3,v}$.

				Now using part (iii)(b) of this proposition related to the (split) multiplicative - additive case,  we obtain $\frac{W(E_1'/F_v,\sigma_v\vartheta_{E_1,v})}{W(E_2'/F_v,\sigma_v\vartheta_{E_1,v})} = 
				(-1)^{\langle \sigma_v \vartheta_{E_1,v}, \varkappa_v \rangle + \langle \sigma_v \vartheta_{E_1,v}, \theta_{3,v} \rangle}.$ Notice that $\vartheta_{E_1,v}$ being quadratic (self dual), $W(E_i'/F_v, \sigma_v\vartheta_{E_1,v})=W(E_i/F_v, \sigma_v)$ for $i=1,2$ and also $\langle \sigma_v,  \vartheta_{E_1,v} \chi \rangle = \langle \sigma_v \vartheta_{E_1,v}, \chi \rangle$  for $\chi \in \{\vartheta_{E_1,v}, \varkappa_v\}.$ 
				Moreover, as $p=3$ and $\mu_3 \not\in F_v$, the unramified quadratic character $\varkappa_v=\omega_v$  
				and hence ${\langle \sigma_v, \varkappa_v\vartheta_{E_1,v} \rangle}={\langle \sigma_v, \omega_v\vartheta_{E_1,v} \rangle}$. Hence, we deduce 
				$\frac{W(E_1/F_v,\sigma_v)}{W(E_2/F_v,\sigma_v)} =  
				(-1)^{\langle \sigma_v, \omega_v\vartheta_{E_1,v} \rangle + \langle \sigma_v, \theta_{3,v} \vartheta_{E_1,v}\rangle}.$
				\item[(d)] Let $E_1, E_2$ both have PGA at $v$. Recall, $\rho_{E_i,v}$ is a representation of $\Gal(F_{\cyc,v}(E_i[p])/F_{\cyc,v})$ for $i=1,2$ (see \S \ref{s2}). Using \cite[Page 312]{ser},  $\rho_{E_i}(I_{F_{v}})\cong \overline{\rho}_{E_i}(I_{F_{v}})$ is cyclic of order $e$ with possible value of $e \in \{2, 3, 4, 6\}$. Moreover, $\overline{\rho}_{E_1,v} \cong \overline{\rho}_{E_2,v}$ implies that $e=\# \rho_{E_1}|_{I_{F_v}}=\# \rho_{E_2}|_{I_{F_v}}$. Then for $i=1,2$, there is either one faithful representation $\varepsilon_{E_i,v}$, or two faithful representations $\varepsilon_{E_i,v}, \varepsilon_{E_i,v}^{-1}$ of the inertia subgroup of $\Gal(F_{\cyc,v}(E_i[p])/F_{\cyc,v})$. In either cases, we deduce $\varepsilon_{E_1,v}(-1) = \varepsilon_{E_2,v}(-1)$. Now by Remark \ref{pga}(1) and Theorem \ref{root}(iv)(a), the local root numbers at $v$ are the same.
				\item[(e)]Next, assume  $E_1, E_2$ both are PGNA at $v$. In this case, we have $\rho'_{E_i/F_v}=\rho_{E_i/F_v}$ for $i=1,2$. As observed in Lemma \ref{lem5}(vi), we also have $\rho_{E_1,v} \cong \rho_{E_2,v}$. Note that $I_{F_v} \subset G_{F_{\cyc,v}}$ and $a(\rho_{E_i/F_v} \otimes \sigma)$ just depends on the restriction of $\rho_{E_i/F_v} \otimes \sigma$ to the inertia subgroup.
				Hence using the fact $\rho_{E_1,v} \cong \rho_{E_2,v}$, we obtain $a(\rho_{E_1/F_v} \otimes \sigma)=a(\rho_{E_2/F_v} \otimes \sigma)$. Now by Lemma \ref{frac}, we conclude  $W(E_1/F_v, \sigma_v)=W(E_2/F_v, \sigma_v)$. 
			\end{enumerate}
		\end{enumerate}
		By Lemmas \ref{lem4} - \ref{lem5}, these are the only possibilities of reduction types for the pair $E_1$, $E_2$.  Thus from the above discussion, the result follows.
	\end{proof}
	\subsection{The relation between multiplicities and root numbers}\label{finalrusyl}
	Using Theorem \ref{thm1}, Remark \ref{equal} and Theorem \ref{propo4}, we obtain the main result of this article:
	\begin{theorem} \label{thm4}
		Let $p$ be an odd prime and $k$ be a number field such that all primes of $k$ lying above $p$ are ramified in $k(\mu_p)/k$. Let $E_1, E_2$ be elliptic curves  over $k$ with $E_1[p] \cong E_2[p]$ as $G_k$-modules. Let $\sigma: \De \to \mathrm{GL}_n(\overline{\Q}_p)$ be an irreducible self-dual orthogonal representation. Assume the hypothesis {\bf (H)}. Then 	
		\begin{small}{\[(-1)^{s_p(E_1,\sigma)-s_p(E_2,\sigma)}=
				\frac{W(E_1/F,\sigma)}{W(E_2/F,\sigma)}.
				\]}\end{small}
		\noindent In particular, the $p$-parity conjecture in  \eqref{sigma-parity-state} holds for the twist of $E_1/F$ by $\sigma$ if and only if it holds for the twist of $E_2/F$ by $\sigma$.  
	\end{theorem}
	\begin{remark}\label{symplecompr1}[Necessity of the condition (H)] The hypothesis (H4) stating $S_p(E_i[p^\infty]/K)[p]$ is finite, is necessary in the proof of Theorem \ref{thm2}, in Lemma \ref{L2} and also in   \eqref{896}. In fact, all the 4 hypothesis (H1) - (H4) was used in the proof of Theorem \ref{thm2}. The hypothesis (H1) is also used in the proof of Theorem \ref{propo4}.
	\end{remark}
	\begin{remark}\label{dcjcjkecje}[$\sigma$ symplectic] We have assumed $\sigma $ to be orthogonal in our results. In the other case, when $\sigma$ is (irreducible, self dual and) symplectic, then  assuming  non-degeneracy of the height pairing, the equality  $(-1)^{s_p(E_1,\sigma)-s_p(E_2,\sigma)}= \frac{W(E_1/F,\sigma)}{W(E_2/F,\sigma)}$ is trivially true,  by the following results of Greenberg and Rohrlich:
		\par \cite[Prop. 10.2.3]{MR2807791} Let $E/\Q$ be good ordinary at $p$. Assume  that the $p$-adic height pairing on $S_p(E/K)^\vee\otimes_{\Z_p} \Q_p$ is non-degenerate and $S_p(E/K_\cyc)^\vee$ is a finitely generated torsion $\Lambda=\Z_p[[\mathrm{Gal}(K_\cyc/K)]]$ module. Let  $\sigma$ be a self-dual, irreducible representation of $\Delta$. Then $s_p(E, \sigma)\equiv \lambda_E(\sigma) \pmod 2 $. Moreover, if $\sigma$ is assumed to be symplectic, then $\lambda_E(\sigma)$ is even. 
		\par Under the setting of the above paragraph, when $\sigma$ is symplectic, by \cite[Proposition 2]{MR1387669}, we get the global root number $W(E/F,\sigma)=1$. Thus an analogues of Theorem \ref{thm4} is easily obtained.
	\end{remark}
\begin{remark}[The reverse situation]\label{revsit} It may be natural to  change the setting by fixing one elliptic curve $E$ and consider two self dual representation $\rho_1, \rho_2 :\De \lra \mathrm{GL}_n(\mathcal F)$,  such that $\tilde{\rho_1}\cong \tilde{\rho_2}$, where $\tilde{\rho}_i$ is  as defined in \S \ref{sel1}. 
In this setting,  the comparison of  the $p$-parity conjecture for twist of $E$ by $\rho_1$ with that of $E$ by $\rho_2$, has been already established by Greenberg; see \cite[Theorem 3]{MR2807791} for a precise statement.
\end{remark}
\subsection{Examples:}\label{examplessubsec}
	We now discuss numerical examples to demonstrate application of our results. 
Note that the computation over a number field $\neq \Q$ (for example, $p$-descent for $p\neq 2$) is difficult and hence it is not easy { to} find explicit numerical examples.  We thank V. Dokchitser for several helpful suggestions.

	\begin{example}\label{example:1}  Let $E_1, E_2$ be the two elliptic curves from LMFDB:
		\begin{small}{\begin{center}
					$E_1=11.a2: y^2 + y = x^{3} -  x^{2} - 7820 x - 263580  $,
				\end{center}
				\begin{center}
					$E_2=737.a1: y^2 + y = x^{3} -  x^{2} + 406 x - 686.$
		\end{center}}\end{small}
		\noindent
Take $p=3$ and $k=F=\Q$.  
 Note the dimension of  $S_2(\Gamma_0(11)) =1$ and $737=11\cdot 67$, by level lowering (or using SAGE), we can conclude $E_1[3] \cong E_2[3]$ as $G_\Q$-modules.  

Now $E_1$ has good reduction at $3$ and from LMFDB, $a_3(f_{E_1})=-1$, where  $f_{E_1}$ is the newform of weight $2$ attached to $E_1$ via modularity. Thus $E_1$ is $3$-ordinary.
 
 Let $K= 10.10.1559914552888693.1$ (marking as   in LMFDB) be the totally real number field of degree $10$ over $\Q$ with discriminant $K=1093^5$,  which { is} Galois over $\Q$ with $\De:=\text{Gal}(K/\Q)\cong D_5$, the dihedral group of order 10. The ideal class group of $K$ is trivial. Note that $\De $ has two $2$-dimensional self dual, orthogonal, irreducible (faithful) representations and we denote by $\sigma$ any one of them.

		
We  check, via Magma, that  $L(E_1/K,1) \neq 0$.  As $K$ is totally real and the analytic rank of $E_1$ over  $K$ is $0$, we can apply the BSD  over $K$ (\cite[Theorem A]{zh}) to conclude, $E_1(K)$ is finite and $\Sha(E_1/K)$ is also finite. Thus $S_3(E_1/K)$ is finite and consequently $s_3(E_1,\sigma)=0$. On the other hand, $E_1$ has bad reduction only at $q=11$ and $\sigma$ is unramified at $11$. Hence  using \cite[Proposition 10]{MR1387669}, we  deduce $W(E_1/\Q, \sigma)=1$. So the $3$-parity conjecture holds for $(E_1,\Q,\sigma)$.  

Using SAGE, we see that $E_1(K)[3]=0$ and in fact, $E_1(K)[3]$ is an irreducible $G_K$ module.  Moreover, as $S_3(E_1/K)$ is finite,  using a control theorem, we deduce $S_3(E_1/K_\cyc)^\vee$ is a torsion $\Z_3[[\Gamma]]$ module. Thus all the condition of our theorem, except that the Iwasawa $\mu$-invariant of $S_3(E_1/K_\cyc)^\vee$ i.e. $\mu(S_3(E_1/K_\cyc)^\vee)\stackrel{?}=0$, are satisfied.  We also observe  in SAGE that the Tamagawa numbers of $E_1$ over $K$ are $3$-adic units. Further,   $3\nmid \#\tilde{E_1}(\mathbb \kappa)$ as well, where $\kappa$ is the residue field of $K$ at a prime dividing $3$.


Recall,  $\#\Sha_{an}$  is defined to be the conjectural order of the $\Sha$, as predicted by the exact formula  appearing in the strong form of the BSD conjecture.  Using Magma, we compute that $3 \nmid \#\Sha_{an}(E_1/K) $.  Note that in our setting, strong form of the BSD conjecture is not established yet over the totally real field $K$;  even for elliptic curves  with analytic rank $\leq 1$. Also, over $K$, a totally real Galois field of degree 10 over $\Q$, the one half of the Iwasawa Main Conjecture (extension of Kato's divisibility) has not been proven yet. Assuming $3 \nmid \#\Sha_{an}(E_1/K) \implies 3 \nmid \# \Sha(E_1/K)$, we have  $\Sha(E_1/K)[3] =0$. Hence we get that $S_3(E_1/K)=0$ and  then using control theorem and  Nakayama's lemma, we can deduce that $S_3(E_1/K_\cyc)=0$.

Then $(E_1, E_2,\Q,\sigma) $ satisfy all the conditions of Theorem \ref{thm4} and applying the same theorem, we deduce the $3$-parity conjecture holds for $(E_2,\Q,\sigma)$. We check via SAGE, $W(E_2/\Q, \sigma)=1$. Thus, we conclude from Theorem \ref{thm4} that $s_3(E_2,\sigma) \equiv 0 \pmod 2$. 		
		
On the other hand, we have calculated via Magma,  ord$_{s=1}L(E_2/K,s) \geq 2$. As $3 \nmid \#D_5$, we can not deduce the $3$-parity conjecture for $(E_2, \Q, \sigma)$ using regulator constants or existing results in the literature (cf. \cite{MR2534092},  \cite{dd}). We  summarize the example below. (Also see Remark \ref{stresswithmu}.)
\begin{cor}\label{corlast2}
Let $E_1, E_2, K$ be as above. Assume that $3 \nmid \#\Sha(E_1/K)_{an} \implies  3 \nmid\#\Sha(E_1/K).$ Then $3$-parity conjecture holds for $(E_2,\Q, \sigma)$ with $W(E_2/\Q, \sigma)=1$ and $s_3(E_2,\sigma) \equiv 0 \pmod 2$.
\end{cor}
		
\end{example}
\begin{example}\label{example2}	
	Let $E_1, E_2$ be the two elliptic curves as considered in \cite[Page 22]{gv}:
		\begin{small}{\begin{center}
					$E_1=52.a1: y^2 + y = x^{3} + x - 10  $,
				\end{center}
				\begin{center}
					$E_2=364.a1: y^2 + y = x^{3} -584x+5444.$
		\end{center}}\end{small}
		\noindent
Take $p=5$ and $k=F=\Q$.  
 It is shown in \cite{gv} that  $E_1[5] \cong E_2[5]$ as $G_\Q$-modules and $E_1$ has good, ordinary reduction at $5$.
 
 Let $K= 6.6.16974593.1$ (marking as   in LMFDB) be the totally real number field of degree $6$ with discriminant $K=257^3$,  which { is} Galois over $\Q$ with $\De:=\text{Gal}(K/\Q)\cong S_3$. The ideal class group of $K$ is trivial. Note that $\De $ has a $2$-dimensional self dual, orthogonal, irreducible  representation, say $\sigma$.
 
We  check, via Magma, that  $L(E_1/K,1) \neq 0$.  As $K$ is totally real and the analytic rank of $E_1$ over  $K$ is $0$, we can apply the BSD  over $K$ (\cite[Theorem A]{zh}) to conclude, $E_1(K)$ is finite and $\Sha(E_1/K)$ is also finite. Thus $S_5(E_1/K)$ is finite and consequently $s_5(E_1,\sigma)=0$.  On the other hand, $E_1$ has bad reduction  at $q=2, 13$ and  $\sigma$ is unramified at $13$ and  $2$. Hence again using   \cite[Proposition 10]{MR1387669}, we  deduce $W(E_1/\Q, \sigma)=1$. So the $5$-parity conjecture for the twist of $E_1$ by $\sigma$ over $\Q$  holds  true.
 
Using SAGE, we see that $E_1(K)[5]=0$ and in fact, $E_1(K)[5]$ is an irreducible $G_K$ module.  Moreover, as $S_5(E_1/K)$ is finite, via a control theorem, we deduce $S_5(E_1/K_\cyc)^\vee$ is a torsion $\Z_5[[\Gamma]]$ module. Thus all the condition of our theorem, except $\mu(S_5(E_1/K_\cyc)^\vee)\stackrel{?}=0$, are satisfied.  We also observe  in SAGE that the Tamagawa numbers of $E_1$ over $K$ are $5$-adic units. Further, $5\nmid \#\tilde{E_1}(\mathbb \kappa)$, where $\kappa$ is the residue field of $K$ at a prime dividing $5$. 

 
 Using Magma, we have computed that $5 \nmid \#\Sha_{an}(E_1/K) $.  Again, the strong form of BSD conjecture or the  Iwasawa Main Conjecture over $K$ has not been proven yet. Assuming  $5 \nmid \#\Sha_{an}(E_1/K) \implies 5 \nmid \#\Sha(E_1/K)$, we have  $\Sha(E_1/K)[5] =0$. Thus, we obtain $S_5(E_1/K)=0$ and  as in the previous example, using control theorem and  Nakayama's lemma, we  deduce $S_5(E_1/K_\cyc)=0$.

Then $(E_1, E_2,\Q,\sigma) $ satisfy all the conditions of Theorem \ref{thm4} and applying the same theorem, we deduce the $5$-parity conjecture holds for $(E_2,\Q,\sigma)$. Notice that, again  applying \cite[Proposition 10]{MR1387669}, we see that $W(E_2/\Q, \sigma)=-1$. Thus by Theorem \ref{thm4}, we get $s_5(E_2/\Q, \sigma)\equiv 1 \pmod 2$.

On the other hand, we have calculated via Magma,  ord$_{s=1}L(E_2/K,s) = 2$. As $5 \nmid \#S_3$, we can not deduce the $5$-parity conjecture for $(E_2, \Q, \sigma)$ using regulator constants or existing results in the literature (cf. \cite{MR2534092},  \cite{dd}). Thus we have shown:

\begin{cor}\label{corlast}
Let $E_1, E_2, K$ be as above. Assume that $5 \nmid \#\Sha(E_1/K)_{an} \implies  5 \nmid\#\Sha(E_1/K).$ Then $5$-parity conjecture holds for $(E_2,\Q, \sigma)$ with $W(E_2/\Q, \sigma)=-1$ and $s_5(E_2,\sigma) \equiv 1 \pmod 2$.
\end{cor}

\begin{remark}\label{stresswithmu}
We stress that the condition $p \nmid  \#\Sha(E/K)_{an} \implies p \nmid \#\Sha(E/K)$ in Examples \ref{example:1} and \ref{example2} is only used to check the vanishing of  $\mu(S_p(E/K_\cyc)^\vee)$.  According to the  Iwasawa main conjecture for $E$ over a totally real field $K$, the characteristic ideal of $S_p(E/K_\cyc)^\vee$ is given by certain $p$-adic $L$-function of $E$ over $K$, say $\mathcal L_p(E/K)$, which depends on the choice of certain canonical period. In the above two examples, from the $p$-adic BSD formula, one has $\mathcal L_p(E/K)$ is a unit in the Iwasawa algebra $\La$. Let $\mu_{an}$ be defined by $p^{\mu_{an}} \mid \mid \mathcal L_p(E/K)$ in $\La$. Hence the above condition in Corollaries \ref{corlast2} and \ref{corlast} involving  $\#\Sha(E/K)_{an}$ and $\#\Sha(E/K)$ can be replaced by the following condition: 
\begin{equation}\label{lastpakka}
\mu_{an} = 0 \implies \mu(S_p(E/K_\cyc)^\vee)=0.
\end{equation} 
Of course, one half of  the Iwasawa main conjecture of $E$ over $K_\cyc$ stating the characteristic ideal char$_\La(S_p(E/K_\cyc)^\vee) \mid \mathcal L_p(E/K)$ is proved, then the  condition in \eqref{lastpakka} will be satisfied.

There are recent works of \cite{lo}, \cite{hw}, where they have shown in the setting of Hilbert modular forms of parallel weight, the characteristic ideal of  Selmer group over anticyclotomic $\Z_p$ extension divides the corresponding anticylotomic $p$-adic $L$-function.    The condition $p \nmid \#\Sha(E/K)_{an} \implies  p \nmid\#\Sha(E/K)$ can be deduced  from this divisibility.  However, these works are conditional on a form of Ihara's Lemma. The  paper  on arXiv \cite{xi} has removed this dependence on Ihara's Lemma. 
The relation between  $\#\Sha(E/\Q)_{an} $ and  $ \#\Sha(E/\Q)$ is also  discussed in  \cite{bht}.

In any case, as $E_i[p]$ is an irreducible $G_K$-module, it is always expected that  the $\mu(S_p(E_i/K_\cyc)^\vee)=0$ (see discussions following \cite[Theorem 1]{MR2807791}). 
\end{remark}

\begin{remark}\label{infinitelymanyexamples} By results of Rubin-Silverberg \cite{rs}, given an elliptic curve $E_1/\Q$ and for $p \in \{3,5\}$,  there are infinitely many elliptic curves $E/\Q$ such that $E[p] \cong E_1[p]$ as $G_\Q$ modules. Thus, under the assumption on $\#\Sha(E_1/K)[p]$,  Example \ref{example:1} (respectively Example \ref{example2}) shows, there are  infinitely many elliptic curves $E/\Q$ such that the $3$-parity (respectively $5$-parity) conjecture is true for each of $(E,K,\sigma)$. In fact, there are similar phenomenon of congruences for elliptic curves for prime $p=7,11,13$ (cf. \cite{fi}).
\end{remark}
		
\end{example}
Now, we discuss an example where all of our assumptions in Theorem \ref{thm4} (including $\mu=0$)  are (unconditionally) satisfied. 
\begin{example}\label{example:2}  Let $E_1, E_2$ be the two elliptic curves from LMFDB:
		\begin{small}{\begin{center}
					$E_1=56.b1: y^2=x^3-x^2-4$,
				\end{center}
				\begin{center}
					$E_2=392.c1: y^2=x^3-x^2-16x+29$.
		\end{center}}\end{small}
	\end{example}
	\noindent 
	Take $p=3$ and $k=\Q$.  Note that $j(E_1)=-2^{2}\cdot 7^{-1}$. Using SAGE, we check that $E_1[3] \cong E_2[3]$ as $G_\Q$-modules.  Let $\zeta_{19}$ denote a primitive $19$-th roots of unity. Set $F:=\Q(\zeta_{19}+\zeta_{19}^{-1})$  and 
		let $K:=F(\mu_3,m^{1/3})$, where $m$ is any $3$-power free integer.  Let $\sigma$ be an irreducible, self dual, orthogonal representation of $\De=\Gal(K/F)  \cong S_3$. 
		
		 We will now show that all the conditions of Theorem \ref{thm4} are satisfied for $(E_1,E_2, F, \sigma)$ and we will compute the parity of $s_p(E_1,\sigma) -s_p(E_2,\sigma)$ using our results. 
	
	We check from LMFDB that $E_1$ is $3$-ordinary. Now $\Q(\mu_3)_{\cyc} =\Q(\mu_{3^\infty})$ being abelian, $S_3(E_1/\Q(\mu_{3^\infty}))^\vee$ is $\Lambda$-torsion by results of Kato. Also, from \cite[Table B.1]{ddplms} we notice that the  $\mu(S_3(E_1/\Q(\mu_{3^\infty})^\vee)=0$. As, $[K_\cyc: \Q(\mu_{3^\infty})] $ is a power of $3$, the $\mu$-invariant of $E_1$ over $K_\cyc$ also vanishes by \cite[Corollary 3.4]{hama}). Further, from the same table we have $E_1(\Q(\mu_3))[3]=0$. As $[K:\Q(\mu_3)]=27$, we get $E_1(K)[3]=0$ as well.  Then all the conditions of Theorem \ref{thm4} are satisfied. 
	
	The $3$-parity conjecture  for $(E_1,F, \sigma)$ holds by \cite[Theorem 6.13]{dd}. Then applying Theorem \ref{thm4}, we can deduce $3$-parity for $(E_2, F, \sigma)$. However, as $\Delta \cong S_3$ and $p=3$, the $3$-parity conjecture for both $(E_1, F, \sigma ) $ and $(E_2,F,\sigma)$ also  follows from   \cite[Example 6.10]{dd}, using the results  of \cite{dd} on regulator constants.  We  mention that  the $p$-parity conjecture for the twist of $E_2/F$ by $\sigma$ does not follow from \cite[Theorem 6.13]{dd} or  \cite[Theorem 1.3]{MR2534092}.

	This discussion also shows that the $3$-parity conjecture holds for any elliptic curve $E/F$ twisted by $\sigma$ that satisfies $E[3] \cong E_1[3]$ as $G_\Q$-modules.
	
	Now we determine $s_3(E_1,\sigma)- s_3(E_2, \sigma) \pmod 2.$ First, note that the primes $v$ of $F$  dividing $7 $ belong to $\Sigma_0$ and for such a $v$, we have $\delta_{E_1,v} - \delta_{E_2,v} \equiv \langle \sigma_v, 1_v\rangle \pmod 2$ by Lemma \ref{Lem1}
	(ii).   For a prime $v$ of $F$ dividing $2$, $E_1, E_2$ have potentially good, additive reduction.  By Lemma \ref{image}, for $K=F(\mu_3, m^{1/3})$, a prime of $F$ lying above $2$, belongs to $\Sigma_0$ if and only if $m$ is even.  When $v$ in $F$ dividing $2 $ is in $\Sigma_0$,  by Lemma \ref{lem5}, parts (iv), (v) and (vi), $\delta_{E_1,v} - \delta_{E_2,v} \equiv 0 \pmod 2$.   Thus, from the table or Theorem \ref{thm1}, we deduce $s_3(E_1,\sigma)- s_3(E_2, \sigma) \equiv \underset{v\mid 7 \text{ in } F }{\sum}\langle \sigma_v, 1_v\rangle  \pmod 2$. Let us take $u\mid v \mid 7$ where $u, v$ respectively are primes in $K$ and $F$. Since $\mu_3 \subset F_v$, we see that Gal$(K_u/F_v)$ is trivial or has order $3$. If $K_u =F_v$, then $\langle \sigma_v, 1_v\rangle =2$. On the other hand if $[K_u: F_v]=3$, it is easy to see that $\langle \sigma_v, 1_v\rangle  =0$. 
	
	Thus from the above discussion,  for any $K=F(\mu_3, m^{1/3})$, we always get $s_3(E_1,\sigma) \equiv s_3(E_2, \sigma) \pmod 2.$  Similarly, we can deduce $W(E_1/F,\sigma) \equiv W(E_2/F, \sigma).$ 
\section{relation with the theory of `arithmetic local constants'}\label{compmrnjms}
	Mazur-Rubin have developed the theory of `arithmetic local constants' in \cite{mr} which has been used by Nekov\'a\v{r} \cite{ne}. We explain the relation of our method with the theory of arithmetic local constant{{}s}. For simplicity, we consider in this section the special case where $F=K$ and hence $\sigma$ is trivial. However, we remark  about the general case with $\sigma$ non-trivial at the end. 
	\par Let $E_1, E_2$ be two elliptic curves over a number field $F$.  
	We say $E_1[p] \cong E_2[p]$ as   symplectic $G_F$ modules if the $G_F$ linear isomorphism $E_1[p] \cong E_2[p]$  is compatible with the natural, symplectic Weil pairing on $E_i[p]$. Throughout this section, we keep the following hypothesis:\\
	{\bf (Syl)}:   $E_1[p] \cong E_2[p]$ as   symplectic $G_F$ modules.\\
	Under the Hypothesis  (Syl), using Mazur-Rubin's theory of arithmetic local constants, Nekov\'a\v{r} \cite{ne} has expressed the both (i) the ratio of {{}S}elmer corank of $E_1, E_2$ and (ii) the ratio of the local root numbers of $E_1$ and $E_2$ at a finite prime  $v \nmid p$ of $F$ in terms of the arithmetic local constant $d(\mathcal F_v^1, \mathcal F_v^2)$. To state his results more precisely, we need the following definitions.

	For a finite prime $v \nmid p$ of $F$, let  $I_v$ denote the  inertia subgroup of $G_{F_v}$ at $v$. Note that there is  a natural inclusion map from $T_pE_i \lra V_pE_i:= T_pE_i \otimes \Q_p$ and a natural quotient map $T_pE_i \lra \frac{T_pE_i}{pT_pE_i}\cong E_i[p]$. These map induce corresponding   maps on cohomology. Put  (see \cite[\S 2]{ne})
	\begin{small}{
			\begin{equation}\label{apenh1f}
			H^1_f(T_pE_i/F_v) :=\mathrm{ker}\Big(H^1(F_v, T_pE_i) \rightarrow \frac{H^1(F_v, V_pE_i)}{H^1(G_{F_v}/{I_v}, V_pE_i^{I_v})}\Big) \,\, \text{and } \,\,
			\mathcal F_v^i:= \mathrm{Image}\Big(H^1_f(T_pE_i/F_v) \rightarrow H^1(F_v, E_i[p])\Big).
			\end{equation}}\end{small}
	Then $\mathcal F_v^i$ is a finite dimensional $\mathbb F_p$ vector space and  is a Lagrangian subspace of $H^1(F_v, E_1[p]) \cong H^1(F_v, E_2[p])$. For a prime $v \mid p$ of $F$, we omit the definition of $\mathcal F_v^i$ (see \cite[\S 3.2]{ne}). Given any two Lagrangian subspace{{}s} $ \mathcal F_v^1, \mathcal F_v^2$ at a  finite prime $v $ of $F$, define  the arithmetic local constant,
	\begin{small}{\begin{equation}\label{apenalc}
			d(\mathcal F_v^1,\mathcal F_v^2) := \mathrm{dim}_{\mathbb F_p} \frac{\mathcal F_v^1}{\mathcal F_v^1\cap \mathcal F_v^2} \pmod 2=\mathrm{dim}_{\mathbb F_p} \frac{\mathcal F_v^2}{\mathcal F_v^1\cap \mathcal F_v^2} \pmod 2 \in \Z/{2\Z}
			\end{equation}}\end{small}
			
\noindent Recall the following exact  sequence induced by  Kummer theory via the multiplication by $p$ map on $E_i$: 
	\begin{small}{
			\begin{equation}\label{apenkum1}
			0 \lra \frac{E_i(F_v)}{pE_i(F_v)} \stackrel{\kappa^i_v}\lra H^1(F_v, E_i[p]) \lra H^1(F_v, E_i[p^\infty])[p] \lra 0 
			\end{equation} 
	}\end{small}
	For elliptic curves, at $v \nmid p$, $H^1(G_{F_v}/{I_v}, V_pE_i^{I_v})$ always vanishes. Using this, it is easy to see that 
	\begin{small}{\begin{equation}\label{apen39}
			\mathcal F_v^i =  \text{Image}(\kappa^i_v) \cong \frac{E_i(F_v)}{pE_i(F_v)}.
			\end{equation}
	}\end{small} 
\par Assume the  condition (Syl) holds. Then the setting of \cite[\S 2.16]{ne} applies and we deduce from \cite[Theorem 2.17]{ne} that for a finite prime $v \nmid p$ of $F$,
	\begin{small}{\begin{equation}\label{apenrootalc}
			\frac{W(E_1/F_v)}{W(E_2/F_v)}= (-1)^{d(\mathcal F_v^1,\mathcal F_v^2)}.
			\end{equation}}\end{small}
	On the other hand, assume (Syl) and in addition, also assume $E_i[p](F)=0$. Then it follows from \cite[Theorem 1.4]{mr}, \cite[Equ. (4.6.2)]{ne} that 
	\begin{small}{\begin{equation}\label{apenalgbe}
			s_p(E_1) - s_p(E_2) = \underset{v \mid pN^F_{E_1}N^F_{E_2}}{\sum} d(\mathcal F_v^1,\mathcal F_v^2) \pmod2.
			\end{equation}}\end{small}
\begin{remark}\label{remarklast1}We stress that we do not assume this (Syl) hypothesis in Theorem \ref{thm4}. Even if $E_1$ and $E_2$ are sympletically isomorphic i.e. they satisfy (Syl), the analogue of \eqref{apenrootalc} at a prime $v \mid p$ of $F$ is not clear and the situation is more complicated (see \cite[\S 3]{ne} and his subsequent works). \end{remark}

	When $\sigma$ is trivial, we denote $\de_{E_i,v}:=\de_{E_i,v}(\sigma)$, where $\de_{E,v}(\sigma)$ is defined in \S \ref{sel1}.  
	In Theorem \ref{propo4}, we have expressed the ratio of local root numbers $\frac{W(E_1/F_v)}{W(E_2/F_v)}$ in terms of $\delta_{E_1,v} -\delta_{E_2,v} \pmod2$.  On the other hand,    Nekov\'a\v{r} have expressed (in \eqref{apenrootalc})  the same ratio $\frac{W(E_1/F_v)}{W(E_2/F_v)}$, $v\nmid p$ in terms of $d(\mathcal F_{v}^1, \mathcal F_{v}^2)$.  
	Now we will show directly  in Proposition \ref{apencomppro}, without invoking any results of Nekovar or of this article, that under (Syl), at $v \nmid p$ of $F$, the difference of our local Iwasawa invariant $
	\delta_{E_1,v} -\delta_{E_2,v} = d(\mathcal F_{v}^1, \mathcal F_{v}^2) \pmod 2.$ This `completes the triangle' and shows that our methods are compatible with that of Nekov\'a\v{r}.   However to prove Proposition \ref{apencomppro}, we first need the following lemma:
	\begin{lemma}\label{lastlemma}
		Let $E$ be an elliptic curve over $F$ and  $v$ be a  finite prime of $F$ such that $v \nmid p$. Then the natural map $\frac{E(F_{v})}{pE(F_{v})} \lra \frac{E(F_{\cyc,v})}{pE(F_{\cyc,v})}$ is surjective.
	\end{lemma}
		\proof We prove this by considering the reduction type of $E$ at $v$. If $E$ has good reduction at $v \nmid p$, then $E[p^\infty](F_{\cyc,v})$ is $p$-divisible and hence $\frac{E(F_{\cyc,v})}{pE(F_{\cyc,v})} \cong \frac{E[p^\infty](F_{\cyc,v})}{pE[p^\infty](F_{\cyc,v})} =0$ and  we get the required surjectivity trivially. Next, if $E$ has additive reduction at $v$ with $p \geq 5$, then consider $F_v(\mu_p)$, an unramified extension of $F_v$. Then  $E$ continues to be  additive over $F_v(\mu_p)$ and by \cite[Prop. 5.1(iii)]{hama}, $E[p^\infty](F_v(\mu_{p^\infty}))=0$ and hence $E[p^\infty](F_{\cyc,v})=0$. On the other hand, if $\mu_p \subset F_v$ and $E$ has non-split multiplicative reduction at $v$, then $p$ being odd, $E$ continues to have non-split multiplicative reduction over $F_{\cyc,v}$ and  again by  \cite[Prop. 5.1(iii)]{hama}, we get  $E[p^\infty](F_{\cyc,v})=0$.

		Let $F_v=F_{v_0} \subset F_{v_n}  ... \subset F_{\cyc,w}$ be the layers of the (unramified) cyclotomic $\Z_p$ extension of $F$ with $\mathrm{Gal}(F_{v_n}/F_v) \cong \Z/{p^n\Z}$.   To deal with the remaining cases, we make a couple of  observations, to be used.  First, as $\frac{E(F_{\cyc,v})}{pE(F_{\cyc,v})}  $ is finite, the natural map  from $\frac{E(F_{v_n})}{pE(F_{v_n})} {\lra} \frac{E(F_{\cyc,v})}{pE(F_{\cyc,v})}$ is surjective for some sufficiently large $n$. Thus the required surjectivity in the lemma will follow if  we show the surjectivity of $ \frac{E(F_{v})}{pE(F_{v})} \stackrel{b_E}{\lra} \frac{E(F_{v_n})}{pE(F_{v_n})} $ for any $n$. Secondly, via a diagram chase using the multiplication by $p$ map, respectively on $E(F_v)$ and $E(F_{v_n})$,  we observe that the map $b_E$ is surjective if and only if  $\text{dim}_{\F_p} ~ E[p](F_v) = \text{dim}_{\F_p} ~ E[p](F_{v_n})$. In particular, if $E[p](F_v) =0$ then by Nakayama's Lemma,   $E[p](F_{v_n})=0$ as well. Thus we only need to consider the cases when $E[p](F_v) \neq 0$.
		
		If $E$ has additive reduction over $F_v$ then $E[p](F_v) \neq 0$ can happen only when $p=3$. So we consider $p=3$ and $E$ has additive reduction at  $v$. Then recall $E$ attains good reduction over $F_v(E[3])$ \cite[Corollaries 2(b) \& 3]{seta}.  Hence using $F_{v_n}/F_v$ is unramified, we conclude that $\text{dim}_{\F_p} ~ E[p](L) = 2$ is not possible for $L \in \{F_v, F_{v_n}\}$. Hence if $\text{dim}_{\F_p} ~ E[p](F_v) = 1$ then $\text{dim}_{\F_p} ~ E[p](F_{v_n}) = 1$ as well. 
		
		Next, we assume that $\mu_p \not\subset F_v$ and $E$ has  split (respectively non-split) reduction at $v$. Then recall from \S \ref{sel1}, the diagonal entries of $\rho_E\mid_{G_{F_v}}$ are  $\{\omega_v, 1\}$ (respectively $\{\omega_v\varkappa_v, \varkappa_v\}$; where $\varkappa_v$ is a unramified quadratic  character).  Keeping in mind that $F_{v_n}/F_v$ is a $p$-power extension,  $\chi\mid_{G_{F_v}} =1 \Leftrightarrow \chi\mid_{G_{F_{v_n}}} =1$,  where $\chi$ represents any of the above diagonal characters. Thus we see, $\text{dim}_{\F_p} ~ E[p](F_{v_n}) \leq 1$ and it is the same as $\text{dim}_{\F_p} ~ E[p](F_{v}).$ 
		
		Finally, we consider the case where $\mu_p  \subset F_v$ and $E$ is split-multiplicative at $v$. Recall  (\S \ref{sel1}), \begin{small}{$ \rho_{E}\mid_{G_{F_v}} \sim \begin{pmatrix}1 & * \\0 & 1\end{pmatrix}$}\end{small}. If $* \neq 0$ for $\rho_{E}\mid_{G_{F_v}}$ but $*$ becomes $0$  for $\rho_{E}\mid_{G_{F_{v_n}}}$ then some $p$-power root of the Tate parameter $q$ (\cite[Prop 5.1(ii)]{hama}) does not belong to $F_v$ but it is in $F_{v_n}$. However, then $F_{v_n}/F_v$ will be a ramified Kummer extension, which is a contradiction. Thus $* \neq 0$ over $F_v$ if and only if $* \neq 0$ over $F_{v_n}$.  Hence if $\text{dim}_{\F_p} ~ E[p](F_{v}) = 1$  (when $* \neq 0$), then  $\text{dim}_{\F_p} ~ E[p](F_{v_n}) =1$ also. Similarly, when the former rank is 2   (when $* =0$) then latter rank is 2 as well. This completes the proof of the lemma. \qed

	\begin{prop}\label{apencomppro}
		Assume the hypothesis (Syl) stated in this section. Then for   a finite prime   $v  \nmid p$ of $F$, 
		\begin{small}{
				\begin{equation}\label{apencomp}
				\quad \delta_{E_1,v} -\delta_{E_2,v}= d(\mathcal F_v^1, \mathcal F_v^2) \pmod 2.
				\end{equation}}\end{small}
	\end{prop}
		\proof We identify the isomorphic $G_F$ modules $ E_1[p] $ and $E_2[p]$ in this proposition and denote both of them as $E[p]$ in this proof. Recall the unramified cohomology $\mathcal F_v^{un}:=H^1(G_{F_v}/{I_v}, (E[p])^{I_v}) $ is a  Lagrangian subspaces of $H^1(G_{F_v}, E[p])$ \cite[\S 2]{ne}. Then further recall, by a result of Klagsbrun-Mazur-Rubin (see \cite[Prop. 2.9]{ne}),
		\begin{small}{\begin{equation}\label{apensilly1}
				d(\mathcal F_v^1, \mathcal F_v^2) =d(\mathcal F_v^{un}, \mathcal F_v^1) - d(\mathcal F_v^{un}, \mathcal F_v^2)=  \mathrm{dim}_{\mathbb F_p} \frac{\mathcal F_v^{1}}{\mathcal F_v^{un}\cap \mathcal F_v^1}  -  \mathrm{dim}_{\mathbb F_p} \frac{\mathcal F_v^{2}}{\mathcal F_v^{un}\cap \mathcal F_v^2} \pmod 2.
				\end{equation}
		}\end{small}
		We choose a fix a prime $w$ in $F_\cyc$ dividing the prime $v$ in $F$.  As our calculations below does not depend on the choice of $w$ above $v$, by abuse of notation, we will denote $F_{\cyc,w}$ by $F_{\cyc,v}$. By  Kummer theory, we have an exact sequence
		\begin{small}{
				\begin{equation}\label{apenkum2}
				0 \lra \frac{E_i(F_{\cyc,v})}{pE_i(F_{\cyc,v})} \stackrel{\kappa^i_{\cyc,v}}\lra H^1(F_{\cyc,v}, E_i[p]) \lra H^1(F_{\cyc,v}, E_i[p^\infty])[p] \lra 0 
				\end{equation}}\end{small}
		Further, by well known results of Greenberg (cf. \cite{MR2807791}),   $H^1(F_{\cyc,v}, E_i[p])$ is finite and as $v  \nmid p$, $\delta_{E_i,v} = s_v \times \text{dim}_{\F_p} ~ H^1(F_{\cyc,v}, E_i[p^\infty])[p] ,$ where $s_v$ is the number of primes above $v$ in $F_\cyc$ (see discussion above \eqref{relation}). Thus, it is immediate that 
		\begin{small}{\begin{equation}\label{apen1side}
				\delta_{E_1,v} -\delta_{E_2,v} = \text{dim}_{\F_p} ~ \frac{E_1(F_{\cyc,v})}{pE_1(F_{\cyc,v})} - \text{dim}_{\F_p} ~ \frac{E_2(F_{\cyc,v})}{pE_2(F_{\cyc,v})} \pmod 2.
				\end{equation}}\end{small}
		Now we have the commutative diagram: 
		\begin{small}{\begin{equation}  \label{cdlast}
				\setlength{\arraycolsep}{1pt}
				\begin{array}{*{9}c}
				0 &\longrightarrow & \frac{E_i(F_{v})}{pE_i(F_{v})}
				&  \overset{\kappa^i_{v}}\longrightarrow & H^1(F_{v}, E[p])
				& {\longrightarrow} & H^1(F_{v}, E_i[p^\infty])[p] \lra & 0\\
				& & \Big\downarrow{f_i} & & \Big\downarrow{g_i} & & \Big\downarrow{h_i}  & & \\
				0 &\longrightarrow &  \frac{E_i(F_{\cyc,v})}{pE_i(F_{\cyc,v})}
				& \overset{\kappa^i_{\cyc,v}}{\longrightarrow} & H^1(F_{\cyc,v}, E[p]) 
				& \longrightarrow & H^1(F_{\cyc,v}, E_i[p^\infty])[p] \lra & 0.
				\end{array}
				\end{equation}}\end{small}
		\noindent By Lemma \ref{lastlemma}, $f_i$ is surjective in \eqref{cdlast}, for $i=1,2$. So we deduce from \eqref{apen1side}, $\delta_{E_1,v} -\delta_{E_2,v} = \text{dim}_{\F_p} ~ \text{Img}(f_1)  - \text{dim}_{\F_p} ~ \text{Img}(f_2) \pmod 2$, {{} where Img$(f_i)$ stands for the Image$(f_i)$}. Moreover, using \eqref{apen39}, we can express
		\begin{small}{\begin{equation}  \label{apeninchingtowardsend2}
				\delta_{E_1,v} -\delta_{E_2,v} =\text{dim}_{\F_p} ~  \mathcal F_{v}^1 - \text{dim}_{\F_p} ~  \mathcal F_{v}^2 - \big(\text{dim}_{\F_p} ~ \text{ker}(f_1)  - \text{dim}_{\F_p} ~ \text{ker}(f_1)\big) \pmod2.
				\end{equation}}\end{small} 
		Further, we have Gal$(F_v^\mathrm{unr}/F_{\cyc,v}) \cong \underset{\ell \neq p}{\prod}\Z_{\ell}$.  Hence we deduce that the natural restriction map from $H^1(F_{\cyc,v}, E[p]) \rightarrow H^1(I_v, E[p])$ is injective. Consequently, the kernel of $g_i$ is $\mathcal F_{v}^{un} $. Now, notice that 
		\begin{small}{\begin{equation}  \label{apeninchingtowardsend}
				\text{dim}_{\F_p} ~ \text{ker}(f_i) = \text{dim}_{\F_p} ~ \kappa^i_{v}\big(\text{ker}(f_i)\big)=\text{dim}_{\F_p} ~ \Big(\text{ker}(g_i) \cap  \text{Img}(\kappa^i_{v})\Big)= \mathrm{dim}_{\mathbb F_p} {\Big(\mathcal F_{v}^{un}\cap \mathcal F_{v}^i\Big)}.\end{equation}}\end{small} 
		Applying \eqref{apeninchingtowardsend2} and \eqref{apeninchingtowardsend}  in the expression of $d(\mathcal F_{v}^1, \mathcal F_{v}^2)$ in \eqref{apensilly1}, it is immediate   that 
		$$
		\delta_{E_1,v} -\delta_{E_2,v} = d(\mathcal F_{v}^1, \mathcal F_{v}^2) \pmod 2.
		$$
		This completes the proof of the proposition. \qed
	
	\begin{remark}
		In the general case, when $F\neq K$ and $\sigma $ non-trivial, irreducible, self dual representation of $\mathrm{Gal}(K/F)$, then note that if (Syl) is satisfied i.e. $E_1[p] \cong E_2[p]$ as symplectic $G_F$-modules,  then $E_1[p] \otimes \widetilde{\sigma} \cong E_2[p] \otimes \widetilde{\sigma} $ also symplectically as $G_F$-modules. Thus in a similar way to Proposition \ref{apencomppro},  $\delta_{E_1,v}(\sigma) -\delta_{E_2,v}(\sigma) \pmod 2$ can be expressed in terms of arithmetic local constants. 
	\end{remark}

	\par \textbf{Acknowledgements:} S. Jha is supported by SERB ECR and SERB MATRICS grant. T. Mandal acknowledges institute postdoctoral fellowship at  IIT Kanpur. S. Shekhar is supported by DST INSPIRE faculty grant. The  numerical computations were done using   SAGE  and   LMFDB. We thank Vladimir Dokchister for various discussions. 
	\def\cprime{$'$}

\end{document}